\documentclass{lms}
\usepackage{amsmath}
\usepackage{amssymb,amscd}
\usepackage[all]{xy}
\usepackage{epic}
\usepackage[dvips]{graphicx}

\usepackage{enumerate}

\newtheorem{thm}{Theorem}[section]
\newtheorem{theorem}[thm]{Theorem}
\newtheorem{lemma}[thm]{Lemma}
\newtheorem{corollary}[thm]{Corollary}
\newtheorem{proposition}[thm]{Proposition}
\newnumbered{definition}[thm]{Definition}
\newnumbered{remark}[thm]{Remark}
\newnumbered{example}[thm]{Example}
\newnumbered{emp}[thm]{\kern -4pt}

\newcommand{\cA}{\mathcal{A}}
\newcommand{\cE}{\mathcal{E}}
\newcommand{\cO}{\mathcal{O}}
\newcommand{\cX}{\mathcal{X}}
\newcommand{\cW}{\mathcal{W}}
\newcommand{\cZ}{\mathcal{Z}}
\newcommand{\cY}{\mathcal{Y}}
\newcommand{\cJ}{\mathcal{J}}
\newcommand{\rH}{\mathrm H}
\newcommand{\p}{\mathfrak p}
\newcommand{\m}{\mathfrak m}
\newcommand{\D}{\mathfrak D}
\newcommand{\car}{\mathrm{char}}
\newcommand{\expo}{\varepsilon} 
\newcommand{\who}{\widehat{\cO}} 
\newcommand{\wt}{\widetilde} 
\newcommand{\Gal}{\mathop{\rm Gal}}

\newcommand{\Spec}{\mathop{\rm Spec}}
\newcommand{\Proj}{\mathop{\rm Proj}}
\newcommand{\Ann}{\mathop{\rm Ann}}
\newcommand{\Img}{\mathop{\rm Im}} 
\newcommand{\coker}{\mathop{\rm coker}} 

\begin{document}

\title{Congruences of models of elliptic curves}
\author{Qing Liu, Huajun Lu}

\classno{11G07 (primary), 14G20, 14G40, 11G20 (secondary)}

\maketitle

\begin{abstract}
Let $\cO_K$ be a discrete valuation ring with field of fractions $K$
and perfect residue field. 
Let $E$ be an elliptic curve over $K$, let $L/K$ be a finite Galois 
extension and let $\cO_L$ be the integral closure of $\cO_K$ in $L$. 
Denote by $\mathcal X'$ the minimal regular model of 
$E_L$ over $\cO_L$. We show that the special fibers 
of the minimal Weierstrass model and the minimal regular model of 
$E$ over $\cO_K$ are determined by the infinitesimal fiber 
$\mathcal X'_m$ together with the action of $\Gal(L/K)$, when 
$m$ is big enough (depending on the minimal discriminant of $E$ and
the different of $L/K$). 
\end{abstract} 

\maketitle 

\begin{section}{Introduction}

Let $\cO_K$ be a discrete valuation ring with field of fractions $K$
and perfect residue field. 
Let $E$ be an elliptic curve over $K$. The minimal (projective) regular 
model $\cX$ of $E$ over $\cO_K$ encodes interesting 
arithmetical invariants of $E$ (e.g. the conductor of $E$, and the
smooth locus of $\cX$ is the N\'eron model of $E$). It is 
then important to be able to determine this model. 
Let $L/K$ be a finite Galois extension with Galois group $G$ and
let $\cO_L$ be the integral closure of $\cO_K$ in $L$ (this is
a semilocal Dedekind domain). 
Let $\cX'$ be the minimal regular model of $E_L$ over 
$\cO_L$. By the uniqueness of minimal regular models, $G$ acts on 
the $\cO_{K}$-scheme $\cX'$. It is well known that there exists $L$ 
as above such that $\cX'$ is semi-stable. When 
$L/K$ is moreover tamely ramified and $K$ is complete, Viehweg \cite{Vieh} 
(for curves of any genus $\ge 1$) showed that the type of the 
special fiber $\cX_0$ of $\cX$ is determined by the action of $G$ on the
special fiber of $\cX'$. 

In the present work, we consider 
wildly ramified extensions $L/K$ for elliptic curves. We will 
\emph{not suppose} $E_L$ has semi-stable reduction, even though 
this is probably the most interesting situation. 
For any $\cO_{K}$-scheme $\cZ$ and for any integer $N\ge 0$, 
we will denote as usual 
\[\cZ_N:=\cZ\times_{\Spec \cO_{K}} \Spec(\cO_{K}/\pi^{N+1}\cO_K)\]
where $\pi$ is a uniformizing element of $\cO_K$. For any 
$\cO_L$-scheme $\cZ'$, the infinitesimal fiber $\cZ'_N$
is by definition 
\[\cZ'_N:=\cZ'\times_{\Spec \cO_{K}}
\Spec(\cO_{K}/\pi^{N+1}\cO_K).\]
In \S \ref{cte}, Examples \ref{1.4.5} and \ref{2.5.8}, we exhibit for any 
positive integer $l$, two elliptic curves over $K$ having 
isomorphic $(\cX'_l, G)$ but with non-isomorphic special
fibers $\cX_0$. Hence Viehweg's result can not be
extended directly to the wild ramification case. A natural question,
attributed to B. Mazur and pointed out to us by W. McCallum, is whether
$\cX_0$ is determined by $(\cX'_\ell, G)$ for $\ell$ big
enough. We give a positive answer in the present work:
\medskip 

{\sc Theorem \ref{2.5.4}} \ \ {\it Let $N\geq 0$. Let $\Delta$ be 
the minimal discriminant of $E$. Let $\D_{L/K}$ be the different of $L/K$.  
Then the scheme $\cX_{N}$ is determined by $\cX'_{N+\ell}$ and the
action of $G$ on $\cX'_{N+\ell}$ 
for $\ell=2v_{K}(\Delta)+12[v_K(\D_{L/K})]+18$.}
\medskip 

If the reduction type of $E$ is neither $\mathrm{I}_{r}^*$ nor 
$\mathrm{I}_r$ with $r>0$ (e.g. if $E$ has potentially good reduction
and the residue characteristic of $K$ is different from $2$), 
we can find  
such $l$ depending only on $[v_K(\D_{L/K})]$. Note that $[v_K(\D_{L/K})]$
is bounded by a constant depending only on the absolute ramification 
index of $K$ and on the degree $[L:K]$ when $\car(K)=0$. 
\medskip

At this stage, let us make precise the meaning of ``{$\cX_N$ is determined by
$\cX'_{N+\ell}$ and the action of $G$ on $\cX'_{N+\ell}$}''. Let $E_o$
be an elliptic curve 
over $K_o$, let $L_o$ be a finite Galois extension of $K_o$ of the same 
Galois group $G$. Let $\cX_o, \cX_o'$ be the respective minimal regular models 
of $E_o$ over $\cO_{K_o}$ and of $(E_o)_{L_o}$ over $\cO_{L_o}$. 
We say that 
\medskip 

\begin{center}
\text{\emph{$\cX_N$ is determined by $\cX'_{N+\ell}$ and the action of $G$ on $\cX'_{N+\ell}$}}
\end{center}
\medskip

\noindent if the existence of an isomorphism 
$\psi_{N+\ell}: \cO_K/\pi^{N+\ell+1}\cO_K \simeq
\cO_{K_o}/\pi_{o}^{N+\ell+1}\cO_{K_o}$ 
and of $G$-equi\-variant iso\-mor\-phisms 
\[\cO_L/\pi^{N+\ell+1}\cO_L \simeq \cO_{L_o}/\pi_{o}^{N+\ell+1}\cO_{L_o},
\quad \cX'_{N+\ell}\simeq \cX'_{o, N+\ell}\] 
implies the existence of an isomorphism $\cX_N\simeq \cX_{o,N}$
compatible with the isomorphism 
$\cO_K/\pi^{N+1}\cO_K\simeq \cO_{K_o}/\pi_o^{N+1}\cO_{K_o}$ 
induced by $\psi_{N+\ell}$. 
We define similar notion for minimal Weierstrass models. 
\medskip 

Let us present the organization of this paper. In \S \ref{cte} we
construct the examples mentioned above. 
Section \ref{semi-linear} is a technical 
preliminary work. We study the 
invariants of an $\cO_L$-module under a semi-linear 
action. In \S \ref{minimal-w} and \ref{minimal-w_n}, we study the 
minimal Weierstrass model $\cW$ of $E$ over 
$\cO_K$, as well as the fibers $\cW_N$ in relation with the action of 
$G$ on the minimal Weierstrass model of $E_L$ over $\cO_L$. 

In \S \ref{w2r}, we study the relation between $\cW_{N+\ell}$ and 
$\cX_N$. It is known that $\cX$ is obtained by a sequence of 
blowing-ups starting with $\cW$. We show in Theorem~\ref{blup} that 
if $\wt{\cY}\to \cY$ is the blowing-up morphism along a closed 
point in a scheme $\cY$ over $\cO_K$, 
there exists an explicit integer $\ell\ge 0$ such that 
$\cY_{N+\ell}$ determines $(\wt{\cY})_{N}$ for all $N\ge 0$. We
then apply this result to $\cW$ and show that $\cW_{N+\ell}$
determines $\cX_N$ for some explicit constant $\ell$
(Corollary~\ref{2.5.3}). 
\medskip 

The main result Theorem \ref{2.5.4} is proved in \S \ref{cng} using 
the connection  between $\cX$ and the minimal Weierstrass 
model $\cW$ of $E$. The proof can be divided into three steps: 

(1) Let $\cW'$ be the minimal Weierstrass  model of $E_{L}$ over 
$\cO_{L}$. We show that the $G$-action on $\cX'_{N+\ell_1}$ determines the 
$G$-action on $\cW'_{N+\ell_1}$ in Proposition \ref{X_N-W_N}.

(2) We prove that the $G$-action on $\cW'_{N+\ell_{1}}$ determines 
$\cW_{N+\ell_2}$ if $\ell_2 \ll \ell_{1}$ 
(Theorem \ref{2.3.2}). This is the crucial part. We can choose 
$\ell_{1}-\ell_2$ such that it  depends 
only on the valuation of the different of $L/K$. 

(3) Finally, as we decribed above, $\cW_{N+\ell_2}$ determines 
$\cX_N$ for some $\ell_2\geq 0$ (Corollary~\ref{2.5.3}). 
\medskip 

As we always work with pointed schemes $\cW'_N, \cX'_N$, in the last section, 
we show that a Galois invariant section of such a fiber lifts to a 
Galois invariant section over 
$\Spec \cO_L$ 
when $N\gg 0$ (Proposition~\ref{lift-equiv}). 
If we use N\'eron models, then we get an explicit bound on $N$
(Proposition~\ref{G-approx}).  
\smallskip 

We should mention that the present work is similar to 
(and inspired by) Chai-Yu and Chai's articles \cite{CY}, \cite{Chai} 
where they dealt with N\'eron models of semi-abelian varieties, 
though we use a more
down-to-earth method. It is shown in \cite{Chai}, Theorem 7.6, that 
for any semi-abelian variety $A$ over $K$, the infinitesimal fiber $\mathcal A_N$ 
of the finite type N\'eron model $\mathcal A$ of $A$ over $\cO_K$ is determined by 
the $G$-action on $\mathcal A'_{N+\ell}$ (where $\mathcal A'$ is the 
N\'eron model of $A_L$ over $\cO_L$) for $\ell$ big enough and depending
on $\mathcal A'$. 
Related to his work is the computation, in case of elliptic 
curves, of the base change conductor (\ref{base-cc}). 

\begin{acknowledgements}
This work grew from the first part of the second named author's Ph.D thesis. 
He thanks the Institut de Math\'ematiques de Bordeaux for the nice 
working environment and financial support. We thank Jilong Tong for
useful comments on the formal groups of abelian varieties, and 
the referee for a careful reading of the manuscript.
\end{acknowledgements}

\noindent{\bf Convention} Through this work, $K$ will denote a 
discrete valuation field with residue field $k$ of characteristic
$p\ge 0$, 
$\cO_K$ is the valuation ring of $K$, $\pi$ denotes 
a uniformizing element of $K$ and $v_K$ is the normalized valuation
($v_K(\pi)=1$), $E$ is an elliptic curve over $K$ and 
$\cX$ (resp. $\cW$) denotes its minimal projective regular (resp. minimal 
Weierstrass) model over $\cO_K$. 
\medskip 

We will suppose the residue field $k$ is \emph{perfect} 
starting \S \ref{minimal-w_n}.\footnote{We thank Ivan Fesenko for 
encouraging us to remove the original hypothesis $k$ algebraically closed.}
\medskip 

We will denote by $L/K$ a finite Galois extension with Galois group
$G$ and by $\cO_L$ the integral closure of $\cO_K$ in $L$. 
As usual, the different of $L/K$ will be denoted by $\D_{L/K}$. 
The ramification index of $L/K$ at a maximal ideal 
$\p$ of $\cO_L$ will be denoted by $e_{L/K}$. 
The exponent of $\D_{L/K}$ at $\p$ will be denoted by 
$v_L(\D_{L/K})$. As $L/K$ is Galois, these invariants are independent 
on the choice of $\p$. Sometimes it is convenient to 
write $v_K(\D_{L/K}):={v_L(\D_{L/K})}/{e_{L/K}}\in\mathbb Q$.
\end{section}

\begin{section}{Two Examples}\label{cte} 

In this section, we give two examples. The first one shows that,
contrary to the tamely ramified case,  
for any $l\ge 0$, there exist $K$ and $E/K$ such that 
the special fiber $\cX_0$ (resp. $\cW_1$) is not determined by the 
$G$-action on $\cX'_{l}$ (resp. the infinitesimal fiber $\cW'_1$). 
The second example, of similar nature, shows that in equal 
characteristic case, there is no bound on $l$ independent 
on $E$ for which Theorem \ref{2.3.4} holds. 

\begin{example}\label{1.4.5} Let 
$d\ge 3$ be an odd integer divisible by $3$. Let 
$K=W(\overline{\mathbb{F}}_{2})(\pi)$ with
$\pi^{d}=2$  where  $W(\overline{\mathbb{F}}_{2})$ is the Witt
ring of $\overline{\mathbb{F}}_{2}$.
Let $E$ be the elliptic curve defined by  the equation
\[ y^{2}=x^{3}+\pi^3.\]
Then $E$ has good reduction over $L=K(\sqrt{\pi})$, with
\[G=\Gal(L/K)=\langle \sigma\rangle, \quad \sigma(\sqrt{\pi})=-\sqrt{\pi}, \quad
\D_{L/K}=2\sqrt{\pi}\cO_L.\] 
The smooth model
$\cX'$ of $E_L$ over $\cO_L$ is defined by the equation 
 \[v^{2}+v=u^{3},\]
where $x=\pi \sqrt[3]{4} u$ (note that $\sqrt[3]{4} \in K$) and 
$y=\pi^{3/2}(1+2v)$.  Hence the action of $G$ on $\cX'$ is given 
by $\sigma(u)=u$, $\sigma(v)=-1-v$.

Now let $E_o$ be the elliptic curve over the same $K$ defined by the
equation: 
\[ y_o^{2}=x_o^{3}+(1+\pi). \] 
Then $E_o$ has good reduction over $L_o=K(\sqrt{1+\pi})$, with 
\[\Gal(L_o/K)=\langle \sigma\rangle, \quad
\sigma(\sqrt{1+\pi})=-\sqrt{1+\pi}, \quad 
\D_{L_o/K}=2\cO_{L_o}.\]  
The smooth model $\cX'_o$ of $E_o$ over $L_o$ is defined by the
equation: 
\[v_o^{2}+v_o=u_o^{3},\] 
where $x_o=(4(1+\pi))^{1/3}u_o$ with $(1+\pi)^{1/3}\in K$ and 
$y_o=\sqrt{1+\pi}(1+2v_o)$. The action of $G$ on $\cX'_o$ is then
given by $\sigma(u_o)=u_o, \sigma(v_o)=-1-v_o$. 

It is easy to see that we have an isomorphism
\[\cO_{L}/(2)=\cO_{L}/(\pi^{d})\simeq 
{\cO}_{L_o}/(\pi^{d}) \]
which sends $\sqrt{\pi}$ to $\sqrt{1+\pi}-1$ and which is compatible 
with the $G$-action, because modulo $2$, the image of 
$\sigma(\sqrt{\pi})=-\sqrt{\pi}$ is 
$-(\sqrt{1+\pi}-1)=\sigma(\sqrt{1+\pi}-1)+2\equiv\sigma(\sqrt{1+\pi}-1)$. Note that 
$v_K(\D_{L/K})=d+1/2\ne v_K(\D_{L_0/K})=d$, 
hence by 
Lemma~\ref{disc}, $\cO_L/(\pi^{d+1})\not\simeq \cO_{L_0}/(\pi^{d+1})$. 
We also have an isomorphism 
\[\cX'_{d-1} \simeq \cX'_{o,d-1}\] 
which sends $u_o$ (resp. $v_o$) to $u$ (resp. $v$) and which is 
compatible with the $G$-action. However,  the special fibers of the 
minimal regular models of $E$ and $E_o$ over $\cO_{K}$ 
have different Kodaira types: the first curve has type I$^*_{0}$ by 
Tate's Algorithm, and the second one has type II. 
Note that this example doesn't contradict 
the conclusion of Theorem \ref{2.5.4}. 

Let $\cW$ (resp. $\cW_o$) be the minimal Weierstrass model of 
$E$ (resp. $E_o$) over $\cO_K$. 
We have that $\cX', \cX'_o$ are the respective minimal
Weierstrass models over $\cO_L, \cO_{L_o}$. Clearly
the special fibers of $\cW, \cW_o$ are isomorphic, but 
using Lemma \ref{iso-mod-N}, we can show that $\cW_1\not\simeq \cW_{o,1}$. 
\end{example}

\begin{example} \label{2.5.8} Fix $m\ge 1$ and let $r=1, 3$. Let $k$ be an
algebraically closed field of characteristic 2 and $K=k((t))$.
Consider the elliptic curve $_rE:$
\[y^2+t^{3m}y=x^3+t^r,\]
whose $j$-invariant is 0 and whose discriminant has valuation equal to 
$12m$. The above equation defines a minimal Weierstrass model
$_r\cW$ of $_rE$ over $k[[t]]$.  Let $\alpha_{r}$ be a root
of the polynomial $X^2+t^{3m}X+t^r$ in $\overline{K}$ and let
$L_{r}=K[\alpha_{r}]$. Then $_rE_{L_{r}}$ has a smooth
model $_r\cX'$ over $\cO_{L_r}$ defined by the equation:
\[y'^2+y'=x'^3,\] where $x=t^{2m}x'$ and $ y=t^{3m}y'+\alpha_{r}$. Hence
\[G=\Gal(L_{r}/K)=\langle \sigma\rangle,\quad 
\sigma(\alpha_{r})=t^{3m}+\alpha_{r}; \quad \sigma(x')=x',\quad 
\sigma(y')=y'+1.\]
For $r=1$, the model $_1\cW$ is regular, hence $_1E$ has 
reduction type II. 
For $r=3$, the curve $_3E$ has reduction type $I_{0}^*$ by 
Tate's Algorithm. 
Let $d=3m-2$. We have $G$-equivariant isomorphisms
\[\cO_{L_{1}}/(t^{3m-1})\cong \cO_{L_{3}}/(t^{3m-1}), \quad 
(_1\cX')_{d} \cong (_3\cX')_{d}.\] 
Hence the  $l$ in 
Theorem \ref{2.3.4} must be  bigger than $3m-1$ and it tends to
infinity if $m$ does.
\end{example}

\end{section}

\begin{section}{Semi-linear $\cO_L[G]$-modules}\label{semi-linear} 

Let $L/K$ be a finite Galois extension with Galois group $G$. Denote 
the integral closure of $\cO_{K}$ in $L$ by $\cO_{L}$. 
The aim of this section is, given a semi-linear $\cO_L[G]$-module 
$M$, to compare $M^G/\pi^{N+1}M^G$ with the image of $(M/\pi^{N+1+r}M)^G$ 
in $(M/\pi^{N+1}M)^G$ (Proposition \ref{2.1.4}). The result is used in 
\S 4 and \S 5. 

\begin{definition} A \emph{semi-linear $\cO_{L}[G]$-module} is an  
$\cO_{L}$-module $M$ (not necessarily finitely generated) 
endowed with an action of $G$ such that
\begin{enumerate}
\item $g(x_{1}+x_{2})=g(x_{1})+g(x_{2})$ for all $x_{1}, x_{2} \in M$ and 
$g\in G$, 
\item $g(ax)=g(a)g(x)$ for all $a\in \cO_{L}, x\in M$ and 
$g\in G$. 
\end{enumerate}
A \emph{morphism} $\phi$ between two semi-linear $\cO_L[G]$-modules $M$ 
and $N$ is an $\cO_{L}$-morphism which is $G$-equivariant  (i.e.
$\phi(gx)=g\phi(x), \forall g\in G,\forall  x\in M$).
\end{definition}

Let us recall the following well-known lemma (see for instance
\cite{Sen2}, Proposition 1(a) for finite dimensional vector spaces;
the general situation follows easily): 

\begin{lemma} \label{speiser} {\rm ({\bf Speiser's lemma})} 
Let $V$ be a semi-linear $L[G]$-vector space. Then
the canonical morphism of semi-linear $L[G]$-vector spaces
\[L\otimes_K V^G\to V\] 
is an isomorphism.  
\end{lemma}

\begin{proposition}\label{different} Let $M$ be a semi-linear 
$\cO_L[G]$-module. Let 
\[\varphi : \cO_L\otimes_{\cO_K} M^G\rightarrow M\] 
be the natural morphism of semi-linear $\cO_L[G]$-modules.
\begin{enumerate}[{\rm (1)}] 
\item If $M$ is flat over $\cO_L$, then $\varphi$ 
is injective and $\mathrm{rank}_{\cO_K} M^G=\mathrm{rank}_{\cO_L} M$. 
\item The cokernel of $\varphi$ 
is killed by the different ideal $\D_{L/K}$ of $\cO_L$ over $\cO_K$. 
\end{enumerate} 
\end{proposition}

\begin{proof} (1) comes from \ref{speiser} by tensoring 
$\varphi$ by $L$. 

(2) We first reduce to the case when $K$ is complete. 
Let $\who_{L}:=\cO_L\otimes_{\cO_K} \who_K$ and 
let $\hat{\varphi}$ be the canonical map 
\begin{equation}\label{hat-phi}
\hat{\varphi}: \ \who_L\otimes_{\cO_K} M^G\to \who_L\otimes_{\cO_L} M.
\end{equation} 
As $\who_L/\cO_L$ is faithfully flat,  
to prove (2) it is enough to show $\D_{L/K}\coker\hat\varphi=0$. 
Let $\p_1, \dots, \p_n$ be the maximal ideals of $\cO_L$.
Let $D_i$ be the decomposition group of $G$ at $\p_i$. Then 
$\who_L=\oplus_{1\le i\le n} \who_{L, \p_i}$, 
and $\who_{L, \p_i}/\who_K$ is Galois of group $D_i$. 
Let $M_i=\who_{L,\p_i}\otimes_{\cO_L} M$. 
Then (\ref{hat-phi}) can be identified with the direct sum of the maps
\[ \who_{L,\p_i}\otimes_{\who_K} M_i^{D_i} \to M_i.\] 
As $\D_{L/K}\who_L=\oplus_i \D_{\widehat{L}_{\p_i}/\widehat{K}}$, it is
enough to show (2) for the extension $\who_{L,\p_i}/\who_K$ and
the module $M_i$. Hence we can and do suppose $K$ is complete. 
\smallskip 

Now we proceed by induction on $|G|$. 
Suppose that $H\subseteq G$ is a normal subgroup and the
proposition holds for $L/L^H$ and $L^H/K$. Let $E=L^H$. Then 
\[\D_{L/E} M\subseteq \cO_LM^E, \quad 
\D_{E/K} M^E\subseteq \cO_E (M^H)^{G/H}=\cO_E M^G.\] 
As $\D_{L/K}=\D_{L/E}.(\D_{E/K}\cO_L)$, the 
proposition also holds for $L/K$. As $\cO_K$ is complete, 
$\cO_L/\cO_K$ can be decomposed 
into successive Galois monogeneous (cyclic) extensions 
(see for instance the explanations in \cite{BS}, proof of Theorem
4.1). Thus we are reduced to the case $\cO_L$ is monogeneous over $\cO_K$. 
\smallskip 

Let $\cO_L=\cO_K[\theta]$ for some $\theta\in \cO_L$. 
Let $P(X)\in \cO_K[X]$ be the monic minimal polynomial of 
$\theta$. Then $\cO_L\simeq \cO_K[X]/(P(X))$. 
We have a decomposition in $\cO_L[X]$: 
\[P(X)=(X-\theta)f(X), \quad \text{with } f(X)=b_{n-1}X^{n-1}+b_{n-2}X^{n-2}+\dots +b_0\in \cO_L[X].\]
For any $v\in M$, let 
\[g(v)=\sum_{0\le i\le n-1} b_{i}\sum_{\sigma\in G}\sigma(\theta^iv)
\in \cO_L\otimes M^G.\] 
We have
\[g(v)=\sum_{\sigma\in G} \Big(\sum_i b_i \sigma(\theta)^i\Big)\sigma(v)=
\sum_{\sigma\in G}f(\sigma(\theta))\sigma(v)=f(\theta)v=P'(\theta)v.\] 
So $P'(\theta)M \subseteq \cO_L\otimes_{\cO_K} M^G$ and the 
proposition is proved because $\D_{L/K}=P'(\theta)\cO_L$. 
\end{proof}  

\begin{remark}Proposition \ref{different} is sharp 
for monogeneous Galois extensions $\cO_L=\cO_K[\theta]$. Indeed,
let $M=\cO_L[G]$ 
with the natural structure of semi-linear $\cO_L[G]$-module:
\[ \sigma(\sum_{\tau\in G} \lambda_\tau.\tau):=\sum_{\tau\in G} 
\sigma(\lambda_\tau).(\sigma\tau).\] 
This is a right $\cO_L$-module by 
\[ (\sum_{\tau\in G} \lambda_\tau.\tau)*\mu:=
\sum_{\tau\in G} \lambda_\tau\tau(\mu).\tau.\]
Let $b\in \cO_L$ be such that $bM\subseteq \cO_L\otimes M^G$. 
Let us show that $b\in \D_{L/K}$.  
Let $t=\sum_{\tau\in G}1.\tau\in M$.  
The vectors $t*\theta^i\in M$, $0\le i\le n-1$, 
where $n$ is the degree of the minimal polynomial $P(T)$ of $\theta$, 
generate the left $\cO_L$-module $\cO_L\otimes M^G$. So, 
if $e$ denotes the unit of $G$, 
\[b.e=\lambda_0t*e +\lambda_1 t*\theta + \cdots + \lambda_{n-1} t*(\theta^{n-1}),
\quad \lambda_i \in \cO_L.\] 
By expanding $t*\theta^i$, we see that for all $\sigma\ne e$, 
\[ \lambda_0 + \lambda_1\sigma(\theta)+ \cdots + 
\lambda_{n-1}\sigma(\theta)^{n-1}=0.\] 
So 
$F(T)=\lambda_0 + \lambda_1 T + \cdots + \lambda_{n-1}T^{n-1}\in \cO_L[T]$  
is divisible by $f(T):=P(T)/(T-\theta)$, and 
\[b=\lambda_0 + \lambda_1 \theta + \cdots \lambda_{n-1}\theta^{n-1}=
F(\theta)\in f(\theta)\cO_L =\D_{L/K}.\]
\end{remark}

\begin{definition}\label{expo} Let $H$ be an $\cO_{K}$-module. We define 
the \emph{exponent} $\expo(H)$ of $H$ to be, when it exists, 
the smallest non-negative integer $e$ such that $\pi^e H=0$.  
Note that for any $\cO_L$-module $M$, $\expo(M)$ is defined
to be its exponent as $\cO_K$-module. 
\end{definition}

\begin{proposition}\label{2.1.6} 
Let $M$ be a semi-linear $\cO_L[G]$-module, flat over $\cO_L$. 
\begin{enumerate}[{\rm (1)}]
\item Suppose $\car(K)=0$. Let $I$ be the inertia group 
of $G$ at a maximal ideal of $\cO_L$. 
Then $\expo(\rH^1(G, M))\leq v_K(|I|)$ 
(exponent as $\cO_K$-module). 
\item In general, we have 
\[\expo(\rH^1(G, M))\le 2[v_K(\D_{L/K})].\] 
\end{enumerate}
\end{proposition}

\begin{proof} (1) Let $\who_K$ be the completion of $\cO_K$. 
Since $\cO_K\to\who_K$ is flat, we have 
\[\expo(\rH^1(G, M))=\expo(\rH^1(G, M)\otimes_{\cO_K}\who_K)
=\expo(\rH^1(G, \who_K\otimes_{\cO_K} M)).\] 
Let $\mathfrak p_1, \dots, \mathfrak p_r$ be the maximal ideals of 
$\cO_L$. Let $D$ be the decomposition group of 
$\mathfrak p:=\mathfrak p_1$. 
Then 
\[\widehat{M}:=\who_K\otimes_{\cO_K} M=\oplus_{1\le i\le r} 
(\who_{L, \mathfrak p_i}\otimes_{\cO_L} M)
\simeq \mathrm{Ind}_D^G(M_1),\] 
where $M_i=\who_{L, \mathfrak p_i}\otimes_{\cO_L} M$. 
By Shapiro's lemma 
$\rH^1(G, \widehat{M})\simeq 
\rH^1(D, M_1)$. Let $I$ be the inertia group 
at $\mathfrak p$, let $\cO_F=(\who_{L, \p})^I$. 
Then $\cO_F/\who_K$ is \'etale of Galois group $D/I$. The 
inflation-restriction exact sequence 
\[0=\rH^1(D/I, M_1^{I})\to \rH^1(D, M_1) 
\to \rH^1(I, M_1)\] 
implies that $\expo(\rH^1(G, M))\le  
\expo(\rH^1(I, M_1))$. As $I$ is finite, $|I|$ kills $\rH^1(I, M_1)$ 
(\cite{Serre},  VIII.1,  Corollary 1). 
This implies the desired inequality. 

Note that during this reduction step, we didn't change the
valuations of the differents: 
$v_L(\D_{L/F})=v_L(\D_{L/K})$. 
\smallskip

(2) As we saw above, we can suppose that $\cO_K$ is complete 
and $G$ equal to its inertia group. Consider the $G$-equivariant 
exact sequence of $\cO_{L}$-modules: 
\[ 0 \to \cO_{L}\otimes_{\cO_{K}} M^G \longrightarrow M  
\longrightarrow 
M/(\cO_{L}\otimes_{\cO_K} M^G) \to 0.\]
(The exactness at the left comes from Proposition \ref{different}(1)). 
Taking group cohomology, we get the exact sequence
\[ \rH^1(G, \cO_L\otimes_{\cO_K} M^G) \longrightarrow  \rH^1(G,M)
\longrightarrow \rH^1(G, M/(\cO_L\otimes_{\cO_K}M^G)).\]
Let $\D=\D_{L/K}$ and $e=e_{L/K}$. 
By Proposition \ref{different}, we have 
$\D.(M/(M^G\otimes_{\cO_K}\cO_{L}))=0$.
Hence $\D.\rH^1(G, M/(\cO_L\otimes_{\cO_K}M^G))=0$. 
It remains to find the annihilator of $\rH^1(G, \cO_L\otimes_{\cO_K} M^G)$.

If $N$ is a free $\cO_K$-module (with trivial $G$-action), then the
canonical map $\rH^i(G, \cO_L)\otimes_{\cO_K} N\to 
\rH^i(G, \cO_L\otimes_{\cO_K} N)$ is an isomorphism for all $i\ge 0$. 
As $M^G$ is flat over $\cO_K$, it is an increasing union of
free $\cO_K$-modules. This implies easily that 
$\rH^1(G, \cO_L\otimes_{\cO_K} M^G) 
\simeq \rH^1(G,  \cO_L)\otimes_{\cO_K}M^G$, hence
\[ \Ann(\rH^1(G, \cO_L\otimes_{\cO_K} M^G))=
\Ann(\rH^1(G,  \cO_L)\otimes_{\cO_K}M^G)=\Ann(\rH^1(G, \cO_L)).\] 
Let $\pi_L$ be a uniformizing element of $L$. Let us show that 
\begin{equation}\label{kill}
\pi_L^{[v_L(\D)/p]}\rH^1(G, \cO_L)=0
\end{equation} 
where $p$ is the residue characteristic of $K$ (the case $p=0$ is
considered in Part (1)). 
Let $H$ be a normal subgroup of $G$.  
Again using the inflation-restriction exact sequence 
\[ 0 \to \rH^1(G/H, \cO_{L^H})\to \rH^1(G, \cO_L)\to \rH^1(H, \cO_L)\] 
it is easy to show that Equality (\ref{kill}) holds if it holds
for $L^H/K$ and for $L/L^H$. Therefore, similarly to the 
the proof of Proposition \ref{different} (2), we are
reduced to the case when $G$ is cyclic. Using Herbrand's quotient
as in \cite{Sen}, Remark, pp. 38-39 or \cite{LLR}, p. 508, lines 4-9,
we have 
\[
\mathrm{length}_{\cO_K}\rH^1(G, \cO_L)=
\mathrm{length}_{\cO_K} \cO_K/\mathrm{Tr}(\cO_L)=[v_K(\D)].\]
Let $f$ be the degree of the residue extension of $L/K$. Then
\[\mathrm{length}_{\cO_L}\rH^1(G, \cO_L)\le
[[v_L(\D)/e]/f]=[v_L(\D)/ef].\] 
As we can restrict ourselves to non-trivial wild ramified extensions, 
we have $ef=[L:K]\ge p$. Hence $\pi_L^{[v_L(\D)/p]}$ 
kills $\rH^1(G, \cO_L)$. This implies that 
$\pi_L^{v_L(\D)+[v_L(\D)/p]}\rH^1(G, M)=0$ and the
exponent of $\rH^1(G, M)$ is bounded by the smallest integer 
bigger than or equal to $(v_L(\D)+[v_L(\D)/2])/e$. 
The only case this might fail is when $v_L(\D)=e+r$
with $0\le r\le e-1$. As $L/K$ is wildly ramified, this implies
that $L/K$ has no non-trivial intermediate extensions, hence
$G$ is cyclic (of prime order) and $[v_L(\D)/p]=1$. But
then $(v_L(\D)+[v_L(\D)/p])/e\le 2=2[v_K(\D)]$. 
\end{proof}

\begin{remark} When $\car(K)=0$, one has 
(\cite{Sen}, Theorem 3)
\[\expo(\rH^1(G, \cO_L))\le v_K(p)/(p-1).\]
\end{remark} 

For any $\cO_K$-module $F$ and any $N\ge 0$, we will denote 
$F_N = F/\pi^{N+1}F$.  
Let $M$ be a semi-linear $\cO_L[G]$-module flat over 
$\cO_L$. Let $N\ge 0$. We want to compare 
$(M_{N})^G$ with $(M^G)_N$. For all $m\ge N$, we 
have canonical morphisms of $\cO_K$-modules 
\[\xymatrix{ 
(M^G)_m \ar@{^{(}->}[r] \ar@{->>}[d] & (M_m)^G \ar[d]_{f_{m,N}}\\ 
(M^G)_N \ar@{^{(}->}[r]^{f_N} & (M_N)^G. \\
}
\]

\begin{proposition}\label{2.1.4} Let $M$ be a semi-linear 
$\cO_L[G]$-module, flat over $\cO_L$. Consider the integer $h=2[v_K(\D_{L/K})]\ge 0$. Then 
for all $N\ge 0$ 
and for all  $m\ge N+h$,  
$(M^G)_N$ is determined by the $G$-module 
$M_m$. More precisely, $f_N$ induces an isomorphism 
\[(M^G)_N \to \Img f_{m,N}\simeq (M_m)^G/(\pi^{N+1}).\]
\end{proposition}

\begin{proof} The above diagram implies that 
$\Img f_{N}\subseteq \Img f_{m, N}$. 
It remains to show the inverse inclusion.
Let us consider the following commutative diagram with horizontal exact sequen\-ces: 
\[\xymatrix{ 
  0  \ar[r] & M \ar[d]_{\cdot\pi^{m-N}} \ar[r]^{\cdot\pi^{m+1}} & M\ar[d]_{\mathrm{id}} \ar[r] & M_m \ar[d]_{} \ar[r]^{} & 0  \\
  0 \ar[r] & M \ar[r]^{\cdot\pi^{N+1}} & M \ar[r]^{} & M_N \ar[r]^{} & 0  
}
\]
where $\cdot$ means the multiplication and the maps in the rows are the 
canonical surjection. Then we  have the following diagram of  long 
exact sequences by taking group cohomology:
\[\xymatrix{
(M^G)_m \ar[d]_{\mathrm{id}} \ar[r]^{f_{m}} & (M_m)^G \ar[d]_{f_{m,N}}
  \ar[r]^{\Delta_m} & \mathrm{H}^1(G,M)\ar[d]_{\cdot\pi^{m-N}} \ar[r]^{\cdot \pi^{m+1}}& \rH^1(G,M)\ar[d]_{\mathrm{id}} \\
(M^G)_N \ar[r]^{f_{N}} & (M_N)^G \ar[r]^{\Delta_N} & \mathrm{H}^1(G,M) \ar[r]^{\cdot \pi^{N+1}} &
  \rH^1(G,M).}
\]
We see that $\mathrm{Im} f_{m,N}\subseteq \mathrm{Im} f_{N}$ if and only
if $\Delta_N f_{m,N}=\pi^{m-N}\Delta_m=0$. This happens when  
$m-N\ge \expo(\mathrm{H}^1(G,M))$. The latter inequality 
is true by Proposition \ref{2.1.6}. 
\end{proof}

\begin{remark}\label{bound-diff} 
(1) If $L/K$ is tamely ramified, then $h=0$ and we get the well-known
equality $(M_N)^G=(M^G)_N$. 
\smallskip 

(2) Suppose $\mathrm{char}(K)=0$ and
$p>0$. 
Then the ramification filtration of $G$ has length at most 
$v_L(p)/(p-1)$ (\cite{Serre}, IV.2, Exercise 3c). 
Hence $v_{L}(\D_{L/K})\leq |G|v_L(p)/(p-1)$ by \cite{Serre}, 
IV.1, Proposition 4, and we get
\[v_K(\D_{L/K})\le |G|v_K(p)/(p-1).\] 
When $p^2\mid |G|$ or $v_K(p)\ge p-1$, a better bound is given by Lenstra (\cite{Jav},
4.1.1): \[v_K(\D_{L/K})\le v_K(|G|)+{(e_{L/K}-1)}/{e_{L/K}}.\]  
\end{remark}

\begin{corollary} Under the hypothesis of Proposition~\ref{2.1.4}, 
the canonical homomorphism
\[ \varprojlim_n (M^G)_n \to \varprojlim_n (M_n)^G
=(\varprojlim_n M_n)^G\]
is an isomorphism.   
\end{corollary}

\end{section}

\begin{section}{Minimal Weierstrass models}\label{minimal-w} 

We compare in this section the minimal Weierstrass model of $E$ over
$\cO_K$ with that of $E_L$ over $\cO_L$. 

\begin{definition} Let $E$ be an elliptic curve over the field of
fractions of a principal ideal domain $R$ (e.g. $R=\cO_K$ or
$\cO_L$). 
A \emph{Weierstrass model} of $E$ over $R$ is a triplet 
consisting in (1) a scheme $\cW$ proper and flat over $R$  
with geometrically integral fibers; (2) an isomorphism 
$\cW_{\eta}\to E$ where $\cW_\eta$ denotes the generic fiber
of $\cW$ ; and (3) a section $\epsilon\in \cW(R)$ in the
smooth locus of $\cW$ and whose generic fiber is mapped to 
the origin of $E$ by the isomorphism of (2).  To lighten the
notation, we will often omit the isomorphism $\cW_\eta\to E$. 
\end{definition}

Let us recall the correspondence between Weierstrass models and 
Weierstrass equations (\cite{Deligne-T}, \S 1 or \cite{Liu}, \S 9.4.4). 
Let  $\cW$ be a Weierstrass model of $E$ over $R$.  
Consider the invertible sheaf $\cO_{\cW}(n\epsilon)$ on $\cW$. 
Let $L(n\epsilon)=\Gamma(\cW,\cO_{\cW}(n\epsilon))$. 
For all $n\ge 1$, $L(n\epsilon)$ is free of rank $n$ and 
$L((n+1)\epsilon)/L(n\epsilon)$ is free of rank $1$. 
Let $\{1, x\}$,  $\{1, x, y\}$ be respective bases of 
$L(2\epsilon)$ and $L(3\epsilon)$. 
Scaling $x, y$ by suitable units of $R$, we get a relation
\begin{equation} \label{eqK} 
y^2+ (a_{1}x+a_{3})y=x^3+a_{2}x^2+a_{4}x+a_{6}
\end{equation} 
which (after homogenization) defines $\cW$ as a closed subscheme
of ${\mathbb P}^2_{R}$. 
We will call such a triplet $\{1, x, y\}$ a \emph{Weierstrass basis}
of $\cW$. 
Denote by $\Delta(\cW)\in R\setminus \{ 0\}$ the discriminant of the above equation. 

\begin{definition} 
If the ideal $\Delta(\cW)R$ is the biggest possible among all
Weierstrass models of $E$, then $\cW$ 
will be called a \emph{minimal Weierstrass model of $E$}.
Such a model 
exists and is unique up to isomorphism (\cite{Liu}, 9.4.35
or \cite{sil}, VII.1.3.).  
\end{definition}

The notion of Weierstrass model depends {\it a priori} on the 
origin of $E$. However, the next proposition says that, up to 
isomorphism, the choice of the origin does not matter. See also
discussions in \S \ref{G-invariant}. 

\begin{proposition}\label{choice-of-o} Let $(E, e)$ be an
elliptic curve over $K$ with minimal Weierstrass model 
$\cW$ over $\cO_K$. 
\begin{enumerate}[{\rm (1)}] 
\item Let $q\in E(K)$ and let $\cZ$ be the minimal Weierstrass
model of $(E, q)$ over $\cO_K$. Then there exists an isomorphism
$\cW\simeq \cZ$ which maps $e$ to $q$. 
\item Let $N\ge 0$ and let $\epsilon_N$ be the section of $\cW_N$ induced
by $e$. Let $\bar{q}\in \cW_N$ be a section contained in the 
smooth locus. Then there 
exists an isomorphism (not unique) 
$\cW_N\to \cW_N$ which maps $\bar{q}$ to $\epsilon_N$. 
\end{enumerate} 
\end{proposition}

\begin{proof}(1) Let $t : E\to E$ be an isomorphism which maps $e$ to
$q$. Then $\cZ$ endowed with the isomorphism 
$\cZ_\eta\simeq E \overset{t}{\to} E$ 
is a minimal Weierstrass model of $(E, e)$. By the uniqueness property,
we get an isomorphism $\cW\to \cZ$ as desired.  

(2) As $\cO_K$ and $\who_K$ coincide modulo $\pi^{N+1}$, we
can suppose $\cO_K$ is complete. We can lift $\bar{q}$ to a rational 
point $q\in E(K)$. Let $t : E\to E$ be as in (1). Let $\mathcal{E}^0$
be the identity component of the N\'eron model $\mathcal{E}$ of $E$. 
It is equal to the smooth locus of $\cW$. By the universal property
of $\mathcal E$, $t$ extends to a morphism $\mathcal{E}^0\to \mathcal{E}$. 
As $t(e)\in \mathcal{E}^0$, $t$ is actually a morphism
$\mathcal{E}^0\to \mathcal{E}^0\subseteq \cW$. As 
$\cW\setminus \mathcal{E}^0$ has codimension $\ge 2$ in 
$\cW$, $t$ extends to a finite birational morphism $\cW\to \cW$. 
It is an isomorphism because $\cW$ is normal. 
\end{proof}

From now on $\cW$ will denote the {\bf minimal Weierstrass model} of $E$ 
over $\cO_K$. The next lemma can be proved by direct computations using 
\cite{sil}, VII.1.3(d) and  \cite{Deligne-T}, (1.6).

\begin{lemma}\label{2.2.2} Fix a Weierstrass basis  $\{1, x, y\}$ of $\cW$. 
\begin{enumerate}[{\rm (1)}] 
\item Let $w, z\in L(3\epsilon)$. Then $\{1, w, z\}$ is a Weierstrass
basis of $\cW$ if and only if $\{1, w\}$ is a basis of $L(2\epsilon)$, 
$z\in L(3\epsilon)\setminus L(2\epsilon)$ and $z^2-w^3\in L(5\epsilon)$. 
\item The set $\{ 1, w, z\}$ is a Weierstrass basis of some 
Weierstrass model $\cZ$ if and only if 
$w\in L(2\epsilon)\setminus \cO_K$, $z\in L(3\epsilon)\setminus L(2\epsilon)$, 
$z^2-w^3\in L(5\epsilon)$ and $z\in \cO_K+\cO_K.w+\cO_K.y$. 
\item Under the above condition, $w=u^2x+r$, $z=u^3y+u^2sx+t$ for some
$u, r, s, t\in \cO_K$ and 
we have $\Delta(\cZ)=u^{12}\Delta(\cW)$. 
\end{enumerate}
\end{lemma}

Let $L/K$ be a finite Galois extension of $K$ with Galois group 
$G$. Let $\cW'$ be the minimal Weierstrass model of $E_{L}$ over 
$\cO_{L}$ and let $\epsilon'\subset \cW'(\cO_L)$ be the closure in $\cW'$ 
of the origin of $E_{L}$. 
As we saw before, for all $n\ge 1$, $L(n\epsilon')$ is free of rank $n$ and 
$L((n+1)\epsilon')/L(n\epsilon')$ is free of rank $1$ over $\cO_{L}$.
Similarly, $L(n\epsilon')^G$ is 
free of rank $n$ over $\cO_K$ 
and the quotient $L((n+1)\epsilon')^G/(L(n\epsilon')^G)$ 
is free of rank $1$ over $\cO_K$. 

\begin{lemma}\label{2.2.6} The Galois group $G$ acts on the 
$\cO_{K}$-scheme $\cW'$ and induces a semi-linear $G$-action 
{\rm ({\it cf}. $\S 2$)} on $L(n\epsilon')$ for all $n\in\mathbb N$. Moreover,
a subset $\{1, w, z\}\subset L(3\epsilon')$ is a Weierstrass basis 
of some Weierstrass model of $E$ over $\cO_K$ if and only 
if 
\[w\in L(2\epsilon')^G \setminus \cO_K, \quad  z\in L(3\epsilon')^G\setminus L(2\epsilon')^G\]
and 
\[z^2-w^3\in L(5\epsilon'), \quad z\in \cO_L+\cO_Lw+\cO_L y'.\] 
\end{lemma}

\begin{proof} The group $G$ acts on $E_L$ and fixes the generic fiber
of $\epsilon'$. By the uniqueness of the minimal Weierstrass model
over $\cO_{L}$, $G$ acts on $\cW'$ and semi-linearly on 
$L(n\epsilon')$ for all $n\ge 1$. 
The remaining part of the lemma results easily from Lemma~\ref{2.2.2}(2). 
\end{proof}

\begin{theorem}\label{compare-m} Let $v_L$ denote the normalized 
valuation on $L$ associated to a maximal ideal of $\cO_L$.
Let $\cW$, $\cW'$ be the 
respective minimal Weierstrass models of $E$ and $E_L$. Then
\[0\le v_L(\Delta(\cW)) - v_L(\Delta(\cW'))\le 
12(v_L(\D_{L/K})+e_{L/K}-1).\] 
\end{theorem}

\begin{proof} The filtration 
$\cO_K \subseteq L(2\epsilon')^G \subseteq L(3\epsilon')^G$ 
has successive quotients free of rank $1$ over $\cO_K$. 
This implies the existence of a basis $\{1, w_0, z_0\}$ of $L(3\epsilon')^G$ 
over $\cO_K$ such that $\{1, w_0\}$ is a basis of  
$L(2\epsilon')^G$ over $\cO_K$. 
Let $\{ 1, x', y'\}$ be a Weierstrass basis of $\cW'$. 
Then there exist $a_1, \dots, a_5\in \cO_L$, such that 
$w_0=a_{1}x'+a_{2}$, $z_0=a_{3}y'+a_{4}x'+a_{5}$ and 
$a_1, a_3\ne 0$. There exist $\alpha_1, \alpha_3\in K^*$ such 
that $\alpha_1 w_0, \alpha_3z_0\in L(3\epsilon)\otimes K$ 
define a Weierstrass equation of $E$ over $K$. This
implies that $t:=a_{3}^2/a_{1}^3=\alpha_1^3/\alpha_3^2\in K$.  

Now let us construct a Weierstrass model $\cZ$ 
of $E$ over $\cO_K$. If $t\in \cO_K$, set 
\[\beta_{1}=\pi^{2n}t, \ \beta_{3}=\pi^{3n}t\in\cO_K \]
where $n$ is the smallest integer such that $a_{1}|\pi^{n}$. 
If $t^{-1}\in \cO_K$, set 
\[
\beta_{1}=\pi^{2m}t^{-1}, \  \beta_{3}=\pi^{3m}t^{-2}\in\cO_K 
\]
where $m$ is the smallest non-negative integer such that 
$a_{1}|\pi^{m}t^{-1}$. 

Consider $w=\beta_1w_0\in L(2\epsilon')^G$ and $z=\beta_3z_0\in
L(3\epsilon')^G$. We have 
\[w=\beta_1a_1x'+\beta_1a_2, \quad 
z=\beta_3a_3y'+\beta_3a_4x'+\beta_3a_5.\] 
We can check that $(\beta_1a_1)^3=(\beta_3a_3)^2$, and 
$\beta_3/(\beta_1 a_1)=a_1^{-1}(\pi^nt)\in\cO_K$ if $t\in \cO_K$, 
and that $\beta_3/(\beta_1 a_1)=a_1^{-1}(\pi^m t^{-1})\in \cO_K$ otherwise.
Thus $\beta_3a_4 \in \beta_1a_1\cO_L$ and 
$z\in \cO_L+\cO_Lw+\cO_L y'$. By Lemma \ref{2.2.6}, 
$\{1, w, z\}$ is a Weierstrass basis of some Weierstrass model
$\cZ$ of $E$ over $\cO_K$. In particular
$v_K(\Delta(\cZ))\ge v_K(\Delta(\cW))$. 
 
Now let us compute $v_L(\beta_1a_1)$. We have 
\[
\beta_1a_1= \left\lbrace 
\begin{matrix}
(a_1^{-1}\pi^n)^2a_3^2 \hfill & \text{\rm if \ } t\in\cO_K \\
(a_1^{-1}\pi^m t^{-1})^2a_3^2 & \text{\rm otherwise.} 
\end{matrix}
\right.
\] 
By Proposition \ref{different} (2), 
\[\D_{L/K}.L(3\epsilon')\subseteq 
L(3\epsilon')^G\cO_L=\cO_L+\cO_L w_0+ \cO_L z_0.\] 
Hence $v_L(a_3)\le v_L(\D_{L/K})$. Therefore 
\[v_L(\Delta(\cZ))-v_L(\Delta(\cW'))=6v_L(\beta_1a_1)
\le 12(v_L(\D_{L/K})+e_{L/K}-1).\] 
\end{proof}

\begin{corollary}\label{bound-c} 
Keep the notation of {\rm Theorem \ref{compare-m}}. 
Let $\{1, x, y\}$ be a Weierstrass basis of $\cW$ and 
let  $\{1, x', y'\}$ be a Weierstrass basis  of $\cW'$. 
\begin{enumerate}[{\rm (1)}]
\item Then 
\[x=b_1x'+b_2, \quad y=b_3y'+b_4x'+b_5, \quad b_i\in \cO_L\]
with 
\[v_L(b_1)
\le  2(v_L(\D_{L/K})+e_{L/K}-1), \quad 
v_L(b_3)\le 3(v_L(\D_{L/K})+e_{L/K}-1).\]
\item The $\cO_K$-modules $L(2\epsilon')^G/L(2\epsilon)$ and 
$L(3\epsilon')^G/L(3\epsilon)$
are annihilated by $\pi^{2[v_K(\D_{L/K})]+ 3}$ and 
$\pi^{2[v_K(\D_{L/K})]+ 4}$ respectively.
\end{enumerate} 
\end{corollary}

\begin{proof} (1) As $x\in L(2\epsilon')$ and $y\in L(3\epsilon')$, we can write
$x=b_1x'+b_2$, $y=b_3y'+b_4x'+b_5$ for $b_i\in \cO_L$. 
By Lemma \ref{2.2.2}, we have $b_1^3=b_3^2$ and 
$v_L(\Delta(\cW))-v_L(\Delta(\cW'))=6v_L(b_1)$. The bounds on 
$v_L(b_1)$ and $v_L(b_3)$ are then a consequence of 
Theorem~\ref{compare-m} and of the relation $b_3^2=b_1^3$. 

(2) Keep the notation in the proof of \ref{compare-m}. 
As $\beta_1w_0, \beta_3z_0\in L(3\epsilon)$, it is enough to
bound $v_K(\beta_1)$ and $v_K(\beta_3)$. The computations
in the proof of Theorem \ref{compare-m} imply that  
$v_L(\beta_1)\le v_L(\beta_1a_1)\le 2v_L(\D_{L/K})+ 2e_{L/K}-2$ and 
$\beta_3=\beta_1a_1\gamma_1$ with $\gamma_1=\pi^n/a_1$ if 
$t\in \cO_K$ and $\gamma=(\pi^mt)/a_1$ otherwise. 
Hence $v_L(\beta_3)\le v_L(\beta_1)+e_{L/K}-1$. 
Dividing by the ramification index $e_{L/K}$ we then get (2). 
\end{proof}

\begin{corollary} \label{pot-g} Suppose $E$ has good 
reduction over some Galois extension $L$. Then 
\[v_K(\Delta(\cW))\le [12(v_L(\D_{L/K})-1)/e_{L/K}]+12.\]
\end{corollary}

\begin{remark}\label{abs-b} (\emph{Absolute bound for the minimal
discriminant}) Let $f$ be the conductor of $E$ and 
let $m$ be the number of geometric irreducible
components of the special fiber of the minimal
regular model $\cX$ of $E$ over $\cO_K$. Then
Ogg-Saito's formula (\cite{Saito}, Corollary 2, 
\cite{sil2}, \S IV.11) is 
\[v_K(\Delta(\cW))=f+m-1.\] 
Suppose that $E$ has potentially good reduction
and $K$ is a finite extension of 
$\mathbb Q_p$, by \cite{BK}, Theorem 6.2, 
$f\le 2+6v_K(2)+3v_K(3)$. Hence
\[v_K(\Delta(\cW))\le 10+6v_K(2)+3v_K(3).\]
To be more precise, \cite{BK} gives a bound on 
the Artin conductor $f(V, L/K)$ for any $G$-module $V$ 
(\cite{BK}, \S 5), and the bound on $f$ is then deduced 
using the $G$-module of the $\ell$-torsions $E[\ell]$ for a prime
number $\ell\ne p$. Our Corollary \ref{pot-g} 
is of different nature because it gives a bound
of $v_K(\Delta(\cW))$ in terms of the Artin conductor
for the representation $r_G - I_G$ (regular representation 
minus unit representation). It is better when 
$v_L(\D_{L/K})$ is small with respect to $\max\{v_L(2), v_L(3)\}$. 
\end{remark}

\begin{remark}\label{pot-m} Suppose $E$ has potentially multiplicative
  reduction.  
Consider $c_4(\cW)$ the invariant $c_4$ associated to
Equation~(\ref{eqK}). Then $j(E)=c_4(\cW)^3/\Delta(\cW)$. 
Similarly to Theorem~\ref{compare-m}, we have 
$0\le v_L(c_4(\cW)) - v_L(c_4(\cW')) \le 4(v_L(\D_{L/K})+e_{L/K}-1)$. 
Hence 
\[v_K(c_4(\cW))\le [4(v_L(\D_{L/K})-1)/e_{L/K}]+4\] 
for any (quadratic) extension $L/K$ such that $E_L$ has multiplicative
reduction. 
\end{remark}

\begin{emp}\label{base-cc} \emph{\bf Base change conductor}. 
Suppose $K$ henselian.  
In \cite{CY} and \cite{Chai}, C-L. Chai and J-K.  Yu
introduced the notion of base change conductor for algebraic tori
and abelian varieties over $K$. 
Let $L/K$ be a Galois extension such that $E_L$ has semi-stable
reduction. 
Consider the N\'eron model $\mathcal{E}$ 
(resp. $\mathcal{E'}$) of $E$ (resp. $E_{L}$) over $\cO_{K}$ 
(resp. $\cO_{L}$). 
Let $\omega_{\mathcal E/\cO_K}$ (resp. $\omega_{\mathcal E'/\cO_L}$) 
be the module of the translation invariant differential forms on 
$\mathcal{E}$ 
(resp. $\mathcal{E'}$) over $\cO_{K}$ (resp. $\cO_{L}$).
By definition, the \emph{base change conductor} $c(E) \in \mathbb{Q}$ of $E$ 
is the length of 
$(\omega_{\mathcal E/\cO_K}\otimes \cO_{L})/\omega_{\mathcal E'/\cO_L}$ as 
$\cO_{L}$-modules divided by the ramification index $e_{L/K}$ of $L/K$. 
\end{emp}

\begin{proposition} Suppose $K$ henselian. 
Let $\cW$ be the minimal Weierstrass model of $E$ 
over $\cO_{K}$. Then the base change conductor $c(E)$ is given by 
\begin{equation*}
c(E)=\min \{ \frac{1}{12}v_{K}(\Delta(\cW)), \frac{1}{4}v_{K}(c_4(\cW))
\}. 
\end{equation*}
\end{proposition}

\begin{proof} It is known that $\omega_{\mathcal E/\cO_K}$ (resp. 
$\omega_{\mathcal E'/\cO_L}$) is a free 
$\cO_{K}$-module (resp. $\cO_{L}$-module) generated by some 
canonical differential form $\omega=dx/(2y+a_1x)$ 
(resp. $\omega'$) (\cite{Liu}, Proposition 9.4.35). 
By Lemma \ref{2.2.2}, there exists $u\in \cO_L$ such that 
$\omega'=u\omega$ and $\Delta(\cW)=u^{12}\Delta(\cW')$,
$c_4(\cW)=u^4c_4(\cW')$. Hence 
$c(E)=v_{L}(u)/e_{L/K}$. 
If $E$ has potentially good reduction, then $v_{L}(\Delta(\cW'))=0$, 
$c(E)=v_K(\Delta(\cW))/12$ and 
\[\frac{1}{4}v_K(c_4(\cW))-\frac{1}{12}v_K(\Delta(\cW))=
\frac{1}{12}v_K(j)\ge 0.\] 
Similarly, if $E$ has potentially multiplicative reduction, then 
we have $v_L(c_4(\cW'))=0$, $c(E)=v_K(c_4(\cW))/4$ and 
$v_K(c_4(\cW))/4 < v_K(\Delta(\cW))$ because $v_K(j)<0$. 
\end{proof}

\begin{corollary} Let $L/K$ be a finite Galois extension 
such that $E_L$ has semi-stable reduction. Then  
\[c(E)<v_K(\D_{L/K})+1.\] 
In particular, if $\mathrm{char}(K)=0$, then 
$c(E) < 24v_K(p)/(p-1)+1$. 
\end{corollary}

\begin{proof} The first part comes from Corollary~\ref{pot-g} and 
Remark~\ref{pot-m}. For the second part, note that $E$ has semi-stable
reduction over an extension $L/K$ of degree dividing 24 
and then use Remark~\ref{bound-diff}. 
\end{proof}

\end{section}

\begin{section}{Congruences of  minimal Weierstrass models}\label{minimal-w_n}

From now on, we will suppose that $K$ has \emph{perfect residue field}.
For any scheme $\cZ$ over $\cO_K$ (including $\cO_L$-schemes), 
recall that 
\[\cZ_N:=\cZ\times_{\Spec\cO_K}\Spec(\cO_K/\pi^{N+1}\cO_K)\] 
for any non-negative integer $N$. 
Similarly, recall that for any $\cO_K$ or $\cO_L$-module $M$, we denote 
\[M_N=M/\pi^{N+1}M.\] 
For any morphism $f$ of schemes over $\cO_K$, $f_N$ denotes the canonical
morphism obtained by base change to $\cO_K/\pi^{N+1}\cO_K$. 

We keep the notation of \S \ref{minimal-w}. In particular, 
$\cW$ and $\cW'$ are the respective
minimal Weierstrass model of $E$ and $E_L$ over $\cO_K$ and $\cO_L$, 
and $\epsilon'$ is the closure in $\cW'$ of the origin of $E_L$.  
We will also work with another discrete valuation field $K_{o}$
with perfect residue field, a Galois extension $L_o/K_o$ of group $G$
and an elliptic curve $E_o$ over $K_o$. We will denote the analogous 
construction by the same notation with a subscript $o$.  
We will say  \emph{$\cW_{N}$ is determined by the 
$G$-action on $\cW'_{m}$} (or \emph{by $(\cW'_m, G)$} for short) for some 
$m\ge N$ if the existence of compatible $G$-equivariant isomorphisms: 
\[
(\text{\rm Iso}_{m})\quad  \left\lbrace\begin{matrix} 
\cO_{K, m} \simeq \cO_{K_{o}, m}, \hfill \\ 
\theta_{m} : \cO_{L, m} \simeq \cO_{L_o, m}, \hfill \\ 
\cW'_{m}\simeq \cW'_{o, m}, \quad \epsilon'_m \mapsto \epsilon'_{o, m}\hfill 
\end{matrix}
\right.
\] 
implies $\cW_{N}\simeq \cW_{o,N}$. Let us stress that 
$\cW'_{m}\to \cW'_{o, m}$ is supposed to map $\epsilon'_m$ to $\epsilon'_{o,m}$ 
(the Weierstrass models should be regarded as \emph{pointed schemes}). 
The aim of this section is to 
show that $\cW_N$ is determined by $(\cW'_m, G)$ when $m\gg 0$ 
(Theorem \ref{2.3.2}).

\begin{lemma}\label{disc} 
Let $N > v_K(\D_{L/K})-1$. 
If $\cO_{L, N}\simeq \cO_{L_0, N}$ as $G$-modules, then 
$v_{L_o}(\D_{L_o/K_o})=v_{L}(\D_{L/K})$ 
and $e_{L/K}=e_{L_o/K_o}$.
\end{lemma}

\begin{proof} The maximal ideals $\p_1, \dots, \p_n$ 
of $\cO_L$ correspond to the maximal ideals of the 
semi-local ring $\cO_{L, N}$.  
Hence there is a one-one correspondence between the maximal
ideals of $\cO_L$ and that of $\cO_{L_o}$. 
We have 
\[\D_{L/K}\who_{L, \p_i}=\D_{\who_{L, \p_i}/\who_K}\] 
It is therefore enough to deal with the case when $K$ and $K_o$ are complete. 

The isomorphism $\cO_L/\pi\cO_L\simeq \cO_{L_o}/\pi_o\cO_{L_o}$ implies
the equality of ramification indexes $e_{L/K}=e_{L_o/K_o}$. 
Let $G=G_0\supseteq G_1 \supseteq G_2 \supseteq ... \supseteq G_r = \{ 1 \}$
be the ramification filtration of $G$ acting on $\cO_L$ (with 
$G_{r-1}$ non-trivial). Then 
\[v_L(\D_{L/K})=(\sum_{i\ge 0} (|G_i|-1))\ge r\]
(\cite{Serre}, IV.1, Proposition 4\footnote{Here the hypothesis $k$ 
perfect is used. We don't known whether it is really necessary.}). As 
$v_L(\sigma(\pi_L)-\pi_L)\le r$ for all $\sigma\in G\setminus \{1\}$, 
the same property holds for $\pi_{L_o}$ thanks to the
isomorphism $\theta_N$
and because $(N+1)e_{L/K}>v_L(\D_{L/K})\ge r$. Thus the ramification filtration of
$G$ acting on $L_o$ is the same as that of $G$ acting on $L$, and 
$v_L(\D_{L/K})=v_{L_o}(\D_{L_o/K_o})$. 
\end{proof}

Fix $N\ge 0$ and let $m\ge N$. Suppose we are given isomorphisms 
(Iso$_m$) as above. Then they 
induce canonically isomorphisms (Iso$_i$) for all integers $0\le i\le m$. 
The isomorphisms $\cW'_{i} \to \cW'_{o, i}$ ($0\le i\le m$) 
induce $G$-invariant isomorphisms 
\[\varphi_i : L(6\epsilon'_o)_{i} \simeq L(6\epsilon')_{i}, \quad i\le m\]
which respect the filtration by the order of the pole and are compatible
with the multiplications 
$L(n\epsilon'_o)_i\times L(r\epsilon'_o)_i\to L((n+r)\epsilon'_o)_i$. 
For any element $z\in M$ in some $\cO_K$-module, we denote by 
$\bar{z}$ its image in $M_i$ if no confusion is possible. 

\begin{lemma}\label{lift-w}
Let $\{1, x'_o, y'_o\}$ be a Weierstrass basis of 
$\cW'_o$. Then 
$\{1, \varphi_m(\bar{x}'_o), \varphi_m(\bar{y}'_o)\}\subset 
L(3\epsilon')_m$ lifts to a Weierstrass basis of $\cW'$. 
\end{lemma}

\begin{proof}Lift arbitrarily 
$\varphi_m(\bar{x}'_o), \varphi_m(\bar{y}'_o)$ 
to $w'\in L(2\epsilon')$ and $z'\in L(3\epsilon')$. There exist $\lambda_i \in \cO_L$
such that $w'=\lambda_1 x' + \lambda_2$, 
$z'=\lambda_3 y' + \lambda_4 w'+ \lambda_5$ (recall that 
$\{1, x', y'\}$ is a Weierstrass basis of $\cW'$). As $\varphi_m$ is an 
isomorphism, $\lambda_1, \lambda_3\in \cO_L^*$. The relation
$(y_o')^2-(x_o')^3\in L(5\epsilon_o')$ implies that 
$\lambda_3^2-\lambda_1^3=0$ in $\cO_{L, m}$. Therefore 
$\lambda:=\lambda_3^2/\lambda_1^3\in 1+\pi^{m+1}\cO_L$. 
Replacing $w'$ (resp. $z'$) by $\lambda w'$ (resp. $\lambda z'$), 
we find new liftings $w, z$ such that $z^2-w^3\in L(5\epsilon')$
and $\{1, w, z\}$ is a basis of $L(3\epsilon')$. This implies that 
$\{1, w, z\}$ is a Weierstrass basis of $\cW'$.
\end{proof}

The next lemma is used in Example \ref{1.4.5} and in 
Proposition~\ref{lift-min}. 

\begin{lemma} \label{iso-mod-N} 
Let $N\ge 0$, let $\cW, \cW_o$ be respective Weierstrass 
models of $E, E_o$ over $\cO_K$ and let 
$\{ 1, x, y \}$, $\{1, x_o, y_o\}$ be corresponding 
Weierstrass basis. Suppose there exists an isomorphism
$\varphi_N : \cW_N\simeq \cW_{o, N}$. Then the following
properties are true: 
\begin{enumerate}[{\rm (1)}] 
\item there exist $u, s, r\in \cO_{K, N}$ such that
$u\in \cO_{K, N}^*$ and 
\[\varphi_N(\bar{x}_o)=u^2\bar{x}+r, \quad 
\varphi_N(\bar{y}_o)=u^3\bar{y}+u^2s\bar{x}+t.\] 
\item If $N\ge 5$, then $\cW$ is minimal if and only if $\cW_o$ is
minimal. 
\end{enumerate}
\end{lemma}

\begin{proof} (1) is an immediate consequence of 
Lemma~\ref{lift-w} and of Lemma~\ref{2.2.2}(3). 

(2) First the minimality of $\cW$ can be checked over the
strict henselization of $\cO_K$. As $k$ is perfect, we
can suppose $k$ is algebraically closed. 
By Tate's algorithm \cite{Tate}, \S 7-8, $\cW$ is not
minimal if and only if there exist $r, s, t \in \cO_K$ 
such that $v(a_i')\ge i$. This condition is checked
modulo $\pi^6$, so can be detected in $\cW_5$.   
\end{proof}

\begin{lemma} \label{snake} Let 
$0 \to M \to H \to T \to 0$ 
be an exact sequence of $\cO_K$-modules such that $\pi^r T=0$
for some $r\ge 1$. Then for all $m\ge r$, 
$\ker(M_{m} \to H_{m})\subseteq 
\ker(M_{m} \to M_{m-r})$. 
\end{lemma}

\begin{proof} Use the Snake Lemma. 
\end{proof}
 
\begin{theorem}\label{2.3.2} Let $N\ge 0$. If 
\[m\ge N+12[v_K(\D_{L/K})]+19,\] 
then $\cW_{N}$ is determined by the $G$-action on $\cW'_{m}$. 
\end{theorem}

\begin{proof} By hypothesis, we have isomorphisms (Iso$_{m}$)
(hence (Iso$_i$) for all $i\le m$). Denote by 
$\rho_i$ the canonical maps $L(3\epsilon)_i \to (L(3\epsilon')^G)_i$ and
by $\rho_{o, i}$ the analogue maps for $E_o$. We have a commutative
diagram 
\[
\xymatrix{
L(3\epsilon_o)_{m}  \ar[r]^{\rho_{o, m}} & (L(3\epsilon'_o)^G)_{m} 
\ar@{^{(}->}[r] & (L(3\epsilon'_o)_m)^G  \ar[d]^{\varphi_{m}} 
\ar@{^{(}->}[r] & L(3\epsilon'_o)_m \ar[d]^{~\varphi_{m}} \\ 
L(3\epsilon)_{m}  \ar[r]^{\rho_{m}} & (L(3\epsilon')^G)_{m} 
\ar@{^{(}->}[r] & (L(3\epsilon')_m)^G  \ar@{^{(}->}[r] & L(3\epsilon')_m \\ 
}
\]
where the vertical arrows are isomorphisms. We want to
complete this diagram with an isomorphism 
$L(3\epsilon_o)_{m_2}\simeq L(3\epsilon)_{m_2}$ (for some $m_2\le m$)
sending a Weierstrass basis to a Weierstrass basis. 
Let $\{1, x_o, y_o\}$ be a Weierstrass basis of $\cW_o$. 
\smallskip 

{\bf Step 1}. Let $\D=\D_{L/K}$. Let $m_1=m-2[v_K(\D)]$. 
We first construct images of $x_o, y_o$ in $(L(3\epsilon')^G)_{m_1}$. 
According to Proposition \ref{2.1.4}, the above commutative diagram
induces a new commutative diagram at level $m_1$ 
with isomorphic vertical arrows 
\[
\xymatrix{
L(3\epsilon_o)_{m_1}  \ar[r]^{\rho_{o, m_1}} & (L(3\epsilon'_o)^G)_{m_1} 
\ar[d]^{\varphi_{m_1}} 
\ar@{^{(}->}[r] & (L(3\epsilon'_o)_{m_1})^G  \ar[d]^{\varphi_{m_1}}
\ar@{^{(}->}[r] & L(3\epsilon'_o)_{m_1} \ar[d]^{~\varphi_{m_1}} \\ 
L(3\epsilon)_{m_1}  \ar[r]^{\rho_{m_1}} & (L(3\epsilon')^G)_{m_1} 
\ar@{^{(}->}[r] & (L(3\epsilon')_{m_1})^G  \ar@{^{(}->}[r] & L(3\epsilon')_{m_1}  \\ 
}
\]
Let $w\in L(2\epsilon')^G, z\in L(3\epsilon')^G$ be liftings of 
$\varphi_{m_1}(\bar{x}_o)$ and $\varphi_{m_1}(\bar{y}_o)$. 
\smallskip 

{\bf Step 2}. Now we modify $w, z$ to $x, y$ so that 
$\{1, x, y\}$ is a Weierstrass basis of a Weierstrass model 
of $E$ over $\cO_K$. 
We can write 
\[x_o=b_{o,1}x_o'+b_{o, 2}, \quad y_o=b_{o,3}y_o'+b_{o,4}x_o'+
b_{o, 5}, \quad b_{o,i}\in \cO_{L_o}\]
where $\{1, x_o', y_o'\}$ is a Weierstrass basis of $\cW'$. 
By Corollary~\ref{bound-c} (1), Lemmas~\ref{disc} and 
\ref{2.2.2} (3), we have 
\[
v_{L_o}(b_{o,1})\le 2(v_{L_o}(\D_{L_o/K_o})+e_{L_o/K_o}-1)\le m_1,
\quad v_{L_o}(b_{o,1})\le v_{L_o}(b_{o,4}).
\]
Let $\{1, x', y'\}$ be a Weierstrass basis of 
$\cW'$ lifting 
$\{1, \varphi_m(\bar{x}_o'), \varphi_m(\bar{y}_o')\}$ 
(Lemma~\ref{lift-w}). Then 
\begin{equation}\label{eq1} 
w=c_1x'+c_2, \quad z=c_3y'+c_4x'+c_5, 
\quad c_i\in \cO_L. 
\end{equation}
For all $i\le 5$, $\bar{c}_i=\theta_{m_1}(\bar{b}_{o,i})\in 
\cO_{L,m_1}$. Therefore 
\begin{equation}\label{eq2} 
v_{L}(c_{1})\le 2(v_L(\D)+e_{L/K}-1)\le m_1, 
\quad v_{L}(c_{1})\le v_{L}(c_{4}).
\end{equation}
We have 
\[\bar{c}_3^2-\bar{c}_1^3=\theta_{m_1}(\bar{b}_{o,3})^2-
\theta_{m_1}(\bar{b}_{o,1})^3=0\in \cO_{L,m_1}\] 
(Lemma~\ref{2.2.2}(3)). 
Writing $w, z$ in a Weierstrass basis of $E$ (with coefficients
in $K$), we see that $\lambda:=c_3^2/c_1^3\in K$. Moreover, the
inequality on $v_L(c_1)$ implies that   
$v_L(c_1^3)/e <  6[v_K(\D)]+12$, hence 
\[\lambda\in 1+\pi^{m_1+1-[6v_K(\D)]-11}\cO_K.\] 
Let $x=\lambda w$ and $y=\lambda z$ and let 
$m_2=m_1-6[v_K(\D)]-11$. 
Then $y^2-x^3\in L(5\epsilon')$ and $x, y\in L(3\epsilon')^G$ coincide with $w, z$ 
in $(L(3\epsilon')^G)_{m_2}$.
Multiplying Equation (\ref{eq1}) above by $\lambda$,
and using the second inequality of Equation (\ref{eq2}) and 
Lemma~\ref{2.2.6}, we see that $\{1, x, y\}$ is 
a Weierstrass basis of a Weierstrass model $\cZ$ of $E$ over 
$\cO_K$. This implies that $x\in L(2\epsilon), y\in L(3\epsilon)$. 
\smallskip 

{\bf Step 3}. Let us show that $\{1, x, y\}$ is a Weierstrass
basis of $\cW$. The above construction shows that we have a 
canonical commutative diagram 
\[
\xymatrix{
(L(3\epsilon_o')^G)_{m_2} \ar[r]^{\varphi_{m_2}} & (L(3\epsilon')^G)_{m_2} 
\ar[r]^{\varphi^{-1}_{m_2}} & (L(3\epsilon_o')^G)_{m_2} \\ 
\Img(\rho_{o, m_2}) \ar@{^{(}->}[u] \ar[r]& 
\Img(\rho_{m_2}) \ar@{^{(}->}[u] \ar[r]&  
\Img(\rho_{o, m_2}) \ar@{^{(}->}[u] \\ 
L(3\epsilon_o)_{m_2}  \ar@{->>}[u]^{\rho_{m_2}} & 
L(3\epsilon)_{m_2}  \ar@{->>}[u]^{\rho_{m_2}} & 
L(3\epsilon_o)_{m_2}  \ar@{->>}[u]^{\rho_{m_2}}  
}
\]
which implies that 
$\varphi_{m_2}(\Img(\rho_{o, m_2}))=\Img(\rho_{m_2})$. Thus 
$\{1, x, y\}$ generate $L(3\epsilon)$ in $(L(3\epsilon')^G)_{m_2}$. 
As $m_2\ge 2[v_K(\D)]+4$, it follows from 
Corollary~\ref{bound-c}(2) that 
\[\ker(L(3\epsilon)\to (L(3\epsilon')^G)_{m_2})\subseteq 
\pi^{m_2+1}L(3\epsilon')^G=\pi (\pi^{m_2}L(3\epsilon')^G)\subseteq
\pi L(3\epsilon).\] 
By Nakayama's lemma, $\{1, x, y\}$ generate, hence is 
a basis of, $L(3\epsilon)$. Therefore $\{1, x, y\}$ is a 
Weierstrass basis of $\cW$. 
\smallskip 

{\bf Last step}: Let's show $\cW_N\simeq \cW_{o,N}$. 
Let 
\[y_o^2+(a_{o,1}x_o+a_{o,3})y_o=x_o^3+a_{o,2}x_o^2+a_{o,4}x_o+a_{o,6}\] 
be an equation of $\cW_o$. Let $a_i\in \cO_K$ be such that
$\bar{a}_i=\varphi_{m_2}(\bar{a}_{o,i})$. 
Then 
\[y^2+({a}_{1}{x}+a_{3})y=x^3+a_2x^2+a_4x+a_6\] 
holds in $(L(6\epsilon')^G)_{m_2}$. By Lemma~\ref{snake} and 
Corollary~\ref{bound-c}(2), the same relation holds in 
$L(6\epsilon)_{m_3}$, where $m_3=m_2-(4[v_K(\D)]+8)$. 
Therefore $\cW_{m_3}\simeq \cW_{o,m_3}$. As $m_3\ge N$, 
we have $\cW_{N}\simeq \cW_{o, N}$. 
\end{proof}

\begin{remark}The bound on $m$ in \ref{2.3.2} is not optimal. 
Indeed, when $L/K$ is unramified, $\cW_N$ is determined by the
$G$-action on $\cW'_m$ with $m=N$. 
\end{remark}

\begin{proposition}Suppose $E$ has semi-stable reduction over $\cO_K$. 
Then for any $N\ge 0$, $\cW_N$ is determined by $(\cW'_m, G)$ if 
$m\ge 2[v_K(\D_{L/K})]+N$. In particular, $\cW_N$ is determined by 
$(\cW'_N, G)$ if $L/K$ is tamely ramified. 
\end{proposition}

\begin{proof} Suppose we are given a system of isomorphisms
(Iso$_m$). Then the special fiber of $\cW_o$ is semi-stable because
it is isomorphic to the special fiber of $\cW$. Hence $\cW_o$ is 
semi-stable. 
The minimal Weierstrass model $\cW$ commutes with base changes by
Dede\-kind domains when $E$ has semi-stable reduction. 
Therefore $\cW'=\cW_{\cO_L}$, $\cW'_o=(\cW_o)_{\cO_{L_o}}$, 
and 
\[\rH^1(G, L(6\epsilon'))=\rH^1(G, L(6\epsilon)\otimes\cO_L)
=\rH^1(G, \cO_L)\otimes L(6\epsilon)\] 
is killed by $\pi^{2[v_K(\D_{L/K})]}$ (Proposition \ref{2.1.6}). 
Let $m_1=m-2[v_K(\D_{L/K})]\ge N$. As in the proof Theorem \ref{2.3.2}, 
we have a commutative diagram 
\[
\xymatrix{
L(3\epsilon_o)_{m_1}= (L(3\epsilon'_o)^G)_{m_1} 
\ar[d]^{\varphi_{m_1}} 
\ar@{^{(}->}[r] & (L(3\epsilon'_o)_{m_1})^G  \ar[d]^{\varphi_{m_1}}  \\ 
L(3\epsilon)_{m_1}=(L(3\epsilon')^G)_{m_1} 
\ar@{^{(}->}[r] & (L(3\epsilon')_{m_1})^G  \\ 
}
\]
As in Lemma \ref{lift-w}, the image of a Weierstrass basis 
$\{1, x_o, y_o\}$ of $\cW_o$ by $\varphi_{m_1}$ lifts to 
a Weierstrass basis $\{1, x, y \}$ of $\cW$. Therefore 
$\cW_{m_1}\simeq \cW_{o, m_1}$.  
\end{proof}

\begin{theorem}\label{2.3.4}
Suppose $K$ is henselian, $\mathrm{char}(K)=0$, and the residue field 
is algebraically closed of characteristic $p>0$. 
Then there exists a positive integer $l$, 
depending only on the absolute ramification index $v_{K}(p)$, 
such that for any elliptic curve $E$ over $K$, and for $L/K$ the minimal 
extension such that $E_{L}$ has semi-stable reduction, 
$\cW_{N}$ is determined by  the $G$-action on $\cW'_{N+l}$ for
any $N\ge 0$.
\end{theorem}

\begin{proof} The extension $L/K$ is Galois 
(\cite{Des}, th\'{e}or\`{e}me 5.15), and it is well known that 
$|G|=[L:K]$ divides 24. 
The corollary then follows from Remark~\ref{bound-diff}
and Theorem \ref{2.3.2}. 
\end{proof}

\end{section}

\begin{section}{From Weierstrass models to regular models}\label{w2r}

Recall that the residue field of $K$ is perfect (to use Kodaira-N\'eron's
classification, and Ogg-Saito's formula). 
The minimal regular model $\cX$ is the minimal desingularization of 
$\cW$ (\cite{Liu}, Corollary 9.4.37). The models we consider are 
\emph{pointed} with the Zariski closure of the neutral element of 
$E$. In this section, we will prove that $\cW_{N+c}$ determines
$\cX_N$ for an explicit constant $c$ depending on the type of $E$
(Corollary~\ref{2.5.3}). 
The next lemma results from an easy computation on
the tangent spaces. 

\begin{lemma} \label{Y1-sing} 
Let $\cY$ be a scheme flat and locally of finite type over $\cO_K$,
of pure relative dimension $d$ and with regular generic fiber. 
Let $y_0$ be a closed point of the special fiber $\cY_0$ 
of $\cY$. Then $\cY$ is regular 
at $y_0$ if and only if either 
$\dim T_{\cY_0, y_0}=d$, or 
$\dim T_{\cY_0, y_0}=d+1$ and $\pi\in {\m}_{\cY, y_0}^2$. In particular,
the singular locus of $\cY$ is determined by $\cY_1$. 
\end{lemma} 

{\noindent{\bf Notation}} For any $\cO_K$-algebra $A$ and for any 
$n\ge 0$, we denote by 
\[A[\pi^n]=\{ x\in A \mid \pi^n x=0 \}, \quad
A_{\mathrm{tors}}=\cup_{n\ge 1} A[\pi^n].\] 
By convention $\pi^0=1$ and $A[\pi^0]=0$. 
We use similar notation when $A$ is replaced with a sheaf of
$\cO_K$-algebras.
The next lemma is easy. 

\begin{lemma} \label{deform} 
Let $U'$ be a noetherian $\cO_K$-scheme,
and let $U=V(\cJ)$ be the closed subscheme of $U'$ 
defined by a coherent sheaf of ideals $\cJ$. 
Suppose $U$ is flat over $\cO_K$ and 
contains $U'_K$. Then 
\begin{enumerate}[{\rm (1)}] 
\item $\cJ=\cO_{U', \mathrm{tors}}=\cO_{U'}[\pi^c]$ for some 
$c\ge 0$; 
\item the composition of the closed immersions $U_N\to U'_N\to U'_{N+c}$ 
induces an isomorphism 
\[ U_N \simeq V(\cO_{U'_{N+c}}[\pi^c]).\] 
\end{enumerate}
\end{lemma}

Next we will give a bound for $c$ in a specific situation. 

\begin{proposition} \label{constant-c} 
Let $\cZ$ be a noetherian regular scheme, let $\rho: \widetilde{\cZ}\to \cZ$ 
be the blowing-up of $\cZ$ along a closed point $q$, 
let $\cY$ be an integral hypersurface in $\cZ$ 
(thus an effective Cartier divisor) passing through $q$ 
and let $\widetilde{\cY}$ be the
strict transform of $\cY$.
Then the following properties hold: 
\begin{enumerate}[{\rm (1)}] 
\item $\widetilde{\cY}$ is an integral hypersurface in $\widetilde{\cZ}$
and, if $\rho^*\cY$ denotes the pullback of $\cY$ as Cartier divisor, 
\[\rho^*\cY=\widetilde{\cY}+\mu_q(\cY) E\] 
where $\mu_q(\cY)$ is the multiplicity of 
$\cY$ at $q$  and $E$ is the prime exceptional divisor $\rho^{-1}(q)$.
\item Let $\tilde{q}\in \wt{\cY}$ be a closed point lying over $q$. Then 
$\mu_{\tilde{q}}(\wt{\cY})\le \mu_q(\cY)$. 
\item Suppose further that $\cZ$ is an $\cO_K$-scheme, 
$\cY$ is flat over $\cO_K$ and $q$ belongs to the closed fiber of $\cY$. Let $r_E\ge 1$ 
be the multiplicity of $E$ in the special fiber of 
$\wt{\cZ}$ (equal to the multiplicity at $q$ of the special fiber of 
$\cZ$) and let $c=\left \lceil \mu_q(\cY)/r_E \right \rceil$ be the smallest 
integer bigger than or equal to $\mu_q(\cY)/r_E$. Then 
\[\cO_{\rho^*\cY, \mathrm{tors}}=\cO_{\rho^*\cY}[\pi^{c}]\varsupsetneq  
\cO_{\rho^*\cY}[\pi^{c-1}].\]
\end{enumerate}
\end{proposition}

\begin{proof} (1) is well known 
and (2) is a particular case of \cite{Bennett}, Chap. I, Theorem 0.

(3) The decomposition $\rho^*\cY=\wt{\cY}+\mu E$ gives an exact sequence 
\[ 0\to \cO_{\wt{\cZ}}(-\wt{\cY})|_{\mu E}=
\cO_{\wt{\cZ}}(-\wt{\cY})/\cO_{\wt{\cZ}}(-\wt{\cY}-\mu E) 
\to \cO_{\rho^*\cY} \to \cO_{\wt{\cY}}\to 0.\] 
As $\wt{\cY}$ is flat over $\cO_K$, and $\cO_{\wt{\cZ}}(-\wt{\cY})|_{\mu E}$ 
is of torsion (namely killed by $\pi^{c}$ because, by the definition of 
$c$, $\mu E$ is contained in $c$ times the 
closed fiber of $\wt{\cZ}$), we have 
\[\cO_{\rho^*\cY, \mathrm{tors}}=\cO_{\wt{\cZ}}(-\wt{\cY})|_{\mu E}=
\cO_{\rho^*\cY}[\pi^{c}]\]
and $\pi^{c-1}\cO_{\rho^*\cY, \mathrm{tors}}\ne 0$. 
\end{proof}

\begin{theorem} \label{blup} Let $\cY$ be a flat $\cO_K$-scheme of 
finite type and let $q$ be a closed point of $\cY$ contained in the 
closed fiber. Suppose further that $\cY$ is locally at $q$ a 
hypersurface in a regular $\cO_K$-scheme of finite type. Let 
$\wt{\cY}\to \cY$ be the blowing-up along $q$, 
and let $\ell\ge c$ (the constant defined as in \ref{constant-c} (3)). 
Then $\cY_{N+\ell}$ determines $(\wt{\cY})_N$ for all $N\ge 0$. 

More precisely, suppose we have a discrete valuation ring $\cO_{K_o}$ and
$(\cY_o, q_o)$ over $\cO_{K_o}$ with similar properties 
and such that $c_o\le \ell$. If we have an isomorphism 
$\phi_{N+\ell}: \cO_{K,{N+\ell}}\to \cO_{K_o, N+\ell}$ and a
compatible isomorphism 
\[\cY_{N+\ell} \simeq \cY_{o, N+\ell},\]  
then we have an isomorphism 
\[(\wt{\cY})_{N} \simeq (\wt{\cY}_{o})_{N}\]
compatible with the isomorphism $\cO_{K, N}\simeq \cO_{K_o, N}$ 
induced by $\phi_{N+\ell}$. 
\end{theorem}

\begin{proof} We can suppose $\cY$ is singular at $q$ (then $\cY_o$ is also
singular at $q_o$ by Lemma ~\ref{Y1-sing}). 
Then $\dim_{k(q)}T_{\cY, q}=d+1$
if $d=\dim\cO_{\cY,q}$. Let $f_0, \dots, f_d$ be a system of generators 
of $\m_q\cO_{\cY,q}$, let $\cY'$ be the gluing of $\cY\setminus \{ q\}$
and of $\Proj \cO_{\cY,q}[T_0, \dots, T_d]/(f_iT_j-f_jT_i)_{0\le i,
  j\le d}$
along $\Spec(\cO_{\cY,q})\setminus \{ q\}$. 
Then $\wt{\cY}$ is a flat closed subscheme of $\cY'$ with generic fiber 
equal to that of $\cY'$. Therefore $\wt{\cY}=V(\cO_{\cY', \mathrm{tors}})$.  
Using a lifting in $\cO_{\cY_o, q_o}$ of the images of the $f_i$'s in 
$\cO_{(\cY_o)_{N+\ell}, q_o}$, we 
define $\cY'_o$ and clearly we have an isomorphism 
\[ \cY'_{N+\ell} \simeq \cY'_{o, N+\ell} \] 
extending the isomorphism $\cY_{N+\ell} \simeq \cY_{o, N+\ell}$. So, by 
Lemma~\ref{deform}, to show that 
$(\wt{\cY})_N\simeq (\wt{\cY}_o)_N$, it is enough to show that 
$\pi^{\ell} \cO_{\cY', \mathrm{tors}}=0$. This property is local on $\cY$. 
As it trivially holds outside of $q$, it is enough to work with 
a small open neighborhood of $q$ in $\cY$. 

Write locally $\cY=\Spec (C/fC)$ with $(C, \m_C)$ local and regular. Lift $f_0, \dots, f_d$
to $t_0, \dots, t_d\in C$. Since $\m_C/\m_C^2\to \m_q/\m_q^2$ is surjective
and both vector spaces have the same dimension over $k(q)$, this is
an isomorphism and $t_0,\dots, t_d$ is a system of coordinates of $C$. 
Consider 
\[B=C[T_0, \dots, T_d]/(f, \ t_iT_j-t_jT_i)_{i,j}.\] 
Let $\rho : \wt{\cZ}\to \cZ:=\Spec C$ be the blowing-up of $\cZ$ along 
$q$. Then $\cY'=\Proj B=\rho^*\cY$ and Proposition~\ref{constant-c}
shows that $\pi^{\ell}\cO_{\cY', \mathrm{tors}}=0$ and we are done. 
\end{proof}

\begin{remark}
Note that one can not determine $(\wt{\cY})_N$ with a blowup of 
$\cY_{N+\ell}$ because the latter process produces 
a scheme which is birational to $\cY_{N+\ell}$, 
while $(\wt{\cY})_N$ has more irreducible components than
$\cY_{N+\ell}$. 
\end{remark}

\begin{remark}\label{rational-sing}
Let $f:\cX\rightarrow \cW$ be the desingularization morphism of
$\cW$. If $\cW$ is singular, the pre-image of the singular point of 
$\cW$ consists in $(-2)$-curves. It is well known that such
singular points are rational singularities
(one can apply \cite{Art}, Theorem 3, because 
the fundamental cycle $Z$ satisfies $2p_a(Z)-2=Z^2<0$).
Let $\cZ\to \cW$ be the blowing-up of the singular point
of $\cW$. Then $\cZ$ is normal (\cite{Lip2}, Proposition
8.1) and its singular points are rational singularities 
(\cite{Lip2}, Proposition 1.2). Therefore the morphism 
$f : \cX\to \cW$ consists in successive blowing-ups 
\[ \cX=\cW^{(t)} \to \cW^{(t-1)} \to \cdots \cW^{(1)} \to \cW^{(0)}=\cW\]
along (reduced and discrete) singular loci. 
\end{remark}

\begin{corollary} \label{2.5.3} Let $\cW$ be the minimal Weierstrass model 
of $E$. Let $\cX$ be the minimal regular model of $E$ over
$\cO_K$ and let $t$ be the number of blowing-ups defined as
above. Then 
\begin{enumerate}[{\rm (1)}] 
\item $\cW_{N+\ell}$ determines $\cX_N$ (in the sense of 
Theorem~\ref{blup}) if  $\ell\ge 2t+1$. 
\item $t+1$ is bounded by the number of irreducible components of
$\cX_0$. In particular, $t\le 8$ if the 
Kodaira type of $E$ is different from $I_n$ and $I_n^*$.
\item Let $\Delta$ be the minimal discriminant of $E$. Then 
$t\le v_K(\Delta)-1$, except when $E$ has good reduction. 
\end{enumerate}
\end{corollary}

\begin{proof} (1) If $\cW$ is regular, 
then $\cX=\cW$ and there is
nothing to prove. So we suppose $\cW$ is singular.  
The scheme $\cW$ is embedded in $\cZ=\mathbb P^2_{\cO_K}$ as a cubic. 
Around the singular point $q$, $\cW$ is defined by a regular function 
$y^2+(a_1x+a_3)y-(x^3+...)\in \m_{\cZ,q}^2\setminus 
\m_{\cZ, q}^3$. So $\mu_q(\cW)=2$. Let $\ell\ge 1$ be any positive
integer. Applying Theorem~\ref{blup}, we see that $\cW_{N+\ell}$ 
determines ${\cW}^{(1)}_{N+\ell-2}$. As $\cW^{(1)}$ is embedded 
(and has codimension 1) in $\wt{\cZ}$ which is regular,
Proposition~\ref{constant-c} and 
Theorem~\ref{blup} imply that $\cW^{(1)}_{N+\ell-2}$ determines 
$\cW^{(2)}_{N+\ell-4}$. Repeating the same arguments we see that
$\cW_{N+\ell}$ determines $\cW^{(t)}_{N+\ell-2t}$. 
This means that $\cW_{N+\ell}\simeq \cW_{o, N+\ell}$ implies that
$\cW^{(t)}_{N+\ell-2t}\simeq \cW^{(t)}_{o, N+\ell-2t}$. 
Note that by Lemma~\ref{Y1-sing}, the isomorphism 
$\cW^{(i)}_{N+\ell-2i}\simeq \cW^{(i)}_{o, N+\ell-2i}$ maps the 
singular locus of $\cW^{(i)}$ to that of $\cW^{(i)}_o$, so 
$\cW^{(i+1)}_o\to \cW^{(i)}_o$ is the blowing-up of the 
singular locus of $\cW^{(i)}_o$. 

Now taking 
$\ell=2t$ might not be enough (when $N=0$) because we don't know whether
$\cW^{(t)}_{o}$ is the minimal regular model of $E_o$. We have to go 
one step further. Namely if $\cW_{N+2t+1}\simeq \cW_{o, N+2t+1}$, then 
$\cW_{N+1}^{(t)}\simeq \cW^{(t)}_{o, N+1}$.  By Lemma~\ref{Y1-sing}, 
we know that $\cW^{(t)}_{o}$ is regular and $\cW^{(t-1)}_o$ is singular. 
Therefore, $t=t_o$ and $\cW_{o}^{(t_o)}=\cX_o$.  

(2) As each blowing-up $\cW^{(i+1)}\to \cW^{(i)}$ introduces at least one 
irreducible component, we see that $t+1$ is at most equal to 
the number of irreducible components of $\cX_0$.  

(3) This is a direct consequence of (2) and Ogg-Saito's formula. 
\end{proof}

\begin{remark} Tate's algorithm shows that $\cW_6$ determines
whether $E$ has type I$_n$ (for some indeterminate $n\ge 0$), 
II, III, IV, II$^*$, III$^*$, IV$^*$ or $I_n^*$ (for some 
indeterminate $n$). But $\cW_6$ can not determine the value 
of $n$ in general in the case I$_n$ or I$_n^*$. 
Below is a table with more precise value of $t$, the number
of blowing-ups necessary to solve the singularities of $\cW$.

\begin{center}
\begin{tabular}{| c | c | c | c | c | c| c| c| c|}
\hline
{type of} $E$ & I$_0$, I$_1$, II & III, IV & II$^*$ & III$^*$ &
IV$^*$ & I$_n$ & I$_n^*$ \\\hline
value of $t$ & $0$ & $1$  &  $\le 8$ & $\le 7$ & 4 & $[n/2]$ & $\le n+4$ 
\\\hline
\end{tabular} 
\end{center}
\end{remark}

\begin{example} Suppose $\mathrm{char}(k)\ne 2$. Let $n\ge 1$ and 
consider the elliptic curves 
\[E: y^2=(x^2+\pi^{2n+1})(x+1), \quad E_o: y^2=(x^2+\pi^{2n+2})(x+1).\]
We have $\cW_{2n}\simeq \cW_{o,2n}$, but $\cX_0\not\simeq \cX_{o,0}$.  
Here $t=n$ but $t_o=n+1$.  
\end{example}

\end{section}

\begin{section}{Congruences of minimal regular models}\label{cng}

In this section we prove the main theorem of this paper
(Theorem~\ref{2.5.4}). The idea is to show that $\cX'_{N+\ell_1+\ell_2}$
determines $\cW'_{N+\ell_1+\ell_2}$ which 
determines $\cW_{N+\ell_1}$ for some $\ell_2$ and finally that 
$\cW_{N+\ell_1}$ determines $\cX_N$ for some $\ell_1$. 

As $\cW$ is regular in a neighborhood of $\epsilon$, $\cX\to \cW$ 
is an isomorphism above a neighborhood of $\epsilon$. So denote again 
by $\epsilon$ the closure in $\cX$ of the neutral element of $E$. 
The effective Cartier divisor $\epsilon$ on $\cX$ is ample on the 
generic fiber and meets in the 
special fiber $\cX_0$ only in the strict transform $\widetilde{\cW}_0$ 
of the irreducible component of $\cW_0$. Therefore 
$\cW$ is the contraction in $\cX$ of the components different from
$\widetilde{\cW}_0$. By construction, there is a canonical isomorphism 
\begin{equation}\label{W} 
\cW \simeq \Proj (\oplus_{n\ge 0} \rH^0(\cX, \cO_{\cX}(n\epsilon))) 
\end{equation}
(see \cite{BLR}, Theorem 6.7/1). 

\begin{lemma}\label{W_n} 
Let $N\ge 0$. Then 
\[\cW_N \simeq \Proj (\oplus_{n\ge 0} \rH^0(\cX_N, \cO_{\cX}(n\epsilon_N))) \] 
\end{lemma}

\begin{proof}We have to show that the canonical map 
\[\rH^0(\cX, \cO_{\cX}(n\epsilon))_N \to \rH^0(\cX_N, \cO_{\cX}(n\epsilon)_N)\] 
is an isomorphism for all $n\ge 0$. By standard arguments
(\cite{Liu}, Theorem 5.3.20), it is 
enough to show that $\rH^1(\cX, \cO_{\cX}(n\epsilon))_0\to
\rH^1(\cX_0, \cO_{\cX}(n\epsilon_0))$ is surjective. By duality, the
latter is isomorphic to $\rH^0(\cX_0, \cO_{\cX_0}(-n\epsilon_0))=0$ 
and we are done. 
\end{proof}

\begin{proposition}\label{X_N-W_N} Let 
$\cW'$, $\cX'$ be the 
minimal Weierstrass (resp. minimal regular) model of $E_L$ over $\cO_L$. 
Let $N\ge 0$. Then $(\cW'_N, G)$ is determined by 
$(\cX'_N, G)$. 
\end{proposition}

\begin{proof} It is enough to apply the isomorphism (\ref{W}) and the 
previous lemma to the models $\cW'$ and $\cX'$ of $E_L$. Note that 
the isomorphism of \ref{W_n} is compatible with the action of $G$
because $\epsilon$ is invariant by $G$. 
\end{proof}

\begin{theorem}\label{2.5.4} Let $K$ be a discrete valuation field
with perfect residue field. Let $E$ be an elliptic curve over $K$ 
of minimal discriminant $\Delta$. Let $L/K$ be a Galois extension of 
group $G$, of different $\D_{L/K}$.  Then for any $N\ge 0$, the
scheme $\cX_{N}$ is determined by $(\cX'_{N+\ell}, G)$, where 
\[\ell=2v_{K}(\Delta)+12[v_K(\D_{L/K})]+18.\]
\end{theorem}

\begin{proof} By Proposition \ref{X_N-W_N}, $(\cX'_{N+\ell}, G)$ determines
$(\cW'_{N+\ell}, G)$. Theorem \ref{2.3.2} says that the latter determines
$\cW_{N+2{v_K(\Delta)}-1}$. Finally Corollary~\ref{2.5.3} 
implies that $\cX_N$ is determined by the previous data. 
\end{proof}

\begin{corollary}
Suppose $K$ is strictly henselian of mixed characteristics $(0, p)$.
Let $L/K$ be the minimal extension such that $E_L$ has
semi-stable reduction. Then the special fiber 
$\cX_{0}$ is determined by the $G$-action on $\cX'_{\ell_{0}}$ for 
some $\ell_{0}\ge 0$ depending only on  the absolute ramification index $v_K(p)$ of $K$.
\end{corollary}

\begin{proof}If $E$ has potentially multiplicative reduction, 
D. Lorenzini (\cite{Lor}, Theorem 2.8) showed
that $[L:K]\le 2$, $E$ has reduction type $\mathrm{I}_{n+4s}^*$  where 
$n=-v_K(j(E))>0$ and $s=v_L(\D_{L/K})-1\ge 0$. The curve of type
$\mathrm{I}_{r}^*$ is unique up to isomorphism for each $r>0$ 
(\cite{Lu}, Theorem 5.18). 
Hence, using Lemma~\ref{disc}, $\cX_{0}$ is determined by 
$\cX_{\ell_0}'$ with $\ell_0=v_{L}(\D_{L/K})\leq 4v_K(p)/(p-1)$ 
(Remark~\ref{bound-diff}). 

If $E$ has potentially good reduction, then 
$v_{K}(\Delta)\le 12(v_K(\D_{L/K})+1)$ with $[L:K]$ dividing $24$,
hence $v_K(\D_{L/K})\le 24v_K(p)/(p-1)$. 
Then we conclude with Theorem~\ref{2.5.4}. 
\end{proof}

Next we give some inverse results of Theorem~\ref{2.3.2} 
and Theorem~\ref{2.5.4}.

\begin{proposition}\label{W-W'} Let $N\ge v_K(\Delta)$. Then $\cW_N$ determines 
$(\cW'_N, G)$ for any finite Galois extension $L/K$.   
\end{proposition}

\begin{proof} Let $K_o, E_o$ and $L_o/K_o$ be as at the beginning 
of \S \ref{minimal-w_n} and suppose we have isomorphisms 
\[ \cO_{K, N}\simeq \cO_{K_o, N}, \quad 
\cW_N \to \cW_{o,N},\quad 
\theta_N : \cO_{L, N}\simeq \cO_{L_o, N},\] 
the last one being $G$-equivariant. We have to find
a $G$-equivariant isomorphism $\cW'_N\to \cW'_{o,N}$. 

Let $\{1, x, y\}$ (resp. $\{1, x', y'\}$) be a 
Weierstrass basis of $\cW$ (resp. of $\cW'$). 
By Lemma \ref{2.2.2}, we have a
change of coordinates of $E_{L}$:
\[x=u^2x'+r, \quad y=u^3y' + u^2s x'+t, \quad u, r, s, t\in \cO_L.\] 
Let $\phi_N : L(6\epsilon)_N\simeq L(6\epsilon_{o})_N$ be the isomorphism
induced by $\cW_{o,N}\simeq \cW_N$. Let $\{1, x_o, y_o\}$ be a 
Weierstrass basis of $\cW_o$ lifting the image by $\phi_N$ 
of the class of $\{1, x, y\}$ in $L(6\epsilon)_N$. 
Let $u_o, r_o, s_o, t_o\in \cO_{L_o}$ be liftings of the images by
$\theta_N$ of the classes $\bar{u}, \bar{r}, \bar{s}, \bar{t}\in\cO_{L,N}$. Let 
\[x_{o}{\kern -2pt}'=(x_{o}-r_{o})/u_o^2, \quad 
y_{o}{\kern -2pt}'=(y_{o}-u_{o}^2s_{o}x_{o}{\kern -2pt}'-t_{o})/u_o^3\in L(6\epsilon'_o)\otimes
L_o.\]
We claim that $\{1, x'_{o}, y'_{o}\}$ is a Weierstrass basis of 
$\cW'_o$. 
First, the fact that $\{1, x, y\}$ defines a Weierstrass model 
\[ y^2+a_1xy+a_3y=x^3+a_2x^2+a_4x+a_6\] 
over $\cO_L$ implies that 
\[ u \mid a_1+2s, \quad u^2 \mid a_2-sa_1+3r-s^2, ... , \quad 
 u^6 \mid a_6+ra_4+ra_r^3-ta_3-t^2-rta_1\] 
(see \cite{Deligne-T}, page 57, (1.6)). As 
\[6v_L(u)=(v_L(\Delta)-v_L(\Delta'))/2\le v_L(\Delta)/2\le v_L(\pi^{N}),\]  
the above divisibility relations hold in $L(6\epsilon'_o)$. 
Therefore $\{1, x'_o, y'_o\}$ is a Weierstrass basis of some 
Weierstrass model over $\cO_{L_o}$. In particular
\[ v_L(\Delta')=v_L(\Delta)-12v_L(u) =v_{L_o}(\Delta_o)-
12v_{L_o}(u_0) \ge v_L(\Delta'_o).\] 
But by symmetry, $v_L(\Delta'_o)\ge v_L(\Delta')$, so 
the equality holds and the Weierstrass model associated to 
$\{1, x'_o, y'_o\}$ is minimal over $\cO_{L_o}$. 

As the change of variables from $\{1, x, y\}$ to 
$\{1, x', y'\}$ and from $\{1, x_o, y_o\}$ to 
$\{1, x_o', y_o'\}$ are given by 
the same relations modulo $\pi^{N+1}$
(up to $\theta_N$), and $\{1, x, y\}$, $\{1, x_o, y_o\}$ 
define the same equation up to $\cO_{K,N}\simeq \cO_{K_o, N}$, 
we have an isomorphism $\cW'_N\to \cW'_{o,N}$ corresponding 
to $\bar{x}'\to \bar{x}_o{\kern -3pt}'$ and 
$\bar{y}'\to \bar{y}_o{\kern -3pt}'$. 
\end{proof}

\begin{proposition}\label{2.6.3} Let $\Delta$ be the minimal discriminant 
of $E$. Then for any $N\ge 0$, $(\cX'_{N}, G)$ 
is determined by  $\cX_{N+2v_L(\Delta)-1}$. 
\end{proposition}

\begin{proof} Let $\ell=2v_K(\Delta)-1$. First, by Proposition \ref{X_N-W_N} (for $L=K$), 
$\cX_{N+\ell}$ determines $\cW_{N+\ell}$. Second, $\cW_{N+\ell}$ 
determines $(\cW'_{N+\ell}, G)$ by Proposition~\ref{W-W'}. Finally, the
latter determines $(\cX'_{N}, G)$ by similar arguments than
Corollary~\ref{2.5.3} (note that $v_L(\Delta)$ is bigger than 
or equal to the $v_L$ valuation of the minimal discriminant of $E_L$).
\end{proof}

\begin{proposition} \label{lift-min} Suppose 
$\mathrm{char}(K)=p>0$.
Let $N\ge 0$. Then there exists a discrete valuation field $K_{o}$  
of characteristic 0, with residue field equal to $k$, 
and an elliptic curve $E_{o}$ over $K_{o}$ such that 
\[\cO_{K_o, N}\simeq \cO_{K, N}, \quad 
\cW_{o, N}\simeq \cW_{N}, \quad 
\cX_{o, N}\simeq \cX_{N}.\]
\end{proposition}

\begin{proof} Let $n\ge \max\{5, N+2{v_K(\Delta)}-1\}$  
so that $\cX_N$ is determined by $\cW_{n}$ 
(Corollary~\ref{2.5.3}). Let $\cO_{K_{o}}=W(k)[t]/(t^{n+1}-p)$
with uniformizing element $\pi_o=t$. Then 
$\cO_{K_{o},n}\simeq \cO_{K,n}$. 
Lifting $\cW_{n}$ to a Weierstrass equation $\cW_{o}$ over
$\cO_{K_o}$, then we have $v_{K_{o}}(\Delta_{o})=v_{K}(\Delta)$ and
$\cW_{o}$ is minimal by Lemma~\ref{iso-mod-N}(2). 
By construction, $\cW_n\simeq \cW_{o, n}$ (hence $\cW_N\simeq \cW_{o, N}$). 
Again by Corollary~\ref{2.5.3}, we have an isomorphism
$\cX_{N}\simeq \cX_{o, N}$.
\end{proof}

\end{section}

\begin{section}{Lifting equivariant infinitesimal sections} 
\label{G-invariant} 

In our settings, Weierstrass models come with a fixed section. 
But in Proposition \ref{choice-of-o}, we saw that up to isomorphism,
the choice of a section does not really matter. 
We can wonder whether in Theorem \ref{2.3.2} we can dismiss the
given section of $\cW'_m$. Note that, at least in our proof,  we
use the fact that this section of $\cW'_m$ is $G$-equivariant and 
even more, that it extends to a section of $\cW'$ over $\cO_L$ 
induced by a rational point in $E(K)$. Now suppose we are given 
(Iso$_m$) as at the beginning of \S 5 but without the condition 
that the isomorphism $\cW'_m\simeq \cW'_{o,m}$ maps $\epsilon'_m$ 
to $\epsilon'_{o,m}$. The image 
of $\epsilon'_m$ is a $G$-equivariant section of $\cW'_{o,m}$ contained in
the smooth locus of $\cW'_{o,m}$. If for some $m_1\le m$, the image
of $\epsilon'_{m_1}$ in $\cW'_{o, m_1}$ 
extends to a section $Q$ of $\cW'_o$ induced by a rational point
of $q\in E_o(K_o)$ (equivalently, $Q$ is a $G$-equivariant
section of $\cW'_o$), then by Proposition \ref{choice-of-o} we have 
a $G$-equivariant isomorphism $\cW'_{m_1}\simeq \cW'_{o,m_1}$ which
maps $\epsilon'_{m_1}$ to $\epsilon'_{o,m_1}$ and we can apply Theorem \ref{2.3.2}
with $m_1$ instead of $m$. 
See Corollary \ref{lift-equivariant-sections} for some results on $m_1$.  

Let $S$ be a scheme. Let $f: X'\to S'$ be a morphism of $S$-schemes 
and let $H$ be a group acting on the $S$-schemes $X'$, $S'$ 
compatibly with $f$ (in other words, $f$ is $H$-equivariant). 
Then $H$ acts on the set of sections 
$X'(S')$ in the following way: for any section $\rho : S'\to X'$ 
and for any $\sigma\in H$, we put 
$\sigma\star \rho = \sigma\circ \rho \circ \sigma^{-1}\in X'(S')$. 
A section $\rho$ is said \emph{$H$-equivariant} if 
$\sigma\star\rho=\rho$ for all $\sigma\in H$. The set of 
$H$-equivariant sections will be denoted by $X'(S')^H$. 
The above question is to study the image of the canonical map 
\[\cW'(\cO_L)^G\to \cW'(\cO_L/\pi^{m+1}\cO_L)^G.\] 
Suppose from now on that $S'\to S$ is finite and locally free and 
$X'\to S'$ is quasi-projective. Then the Weil restriction $R_{S'/S}X'$ 
exists (\cite{BLR}, Theorem 7.6/4) over $S$ and is endowed
with a canonical action of $H$. Moreover, for any 
$S$-scheme $T$, letting $H$ act trivially on $T$ and
denoting $Y=R_{S'/S}X'$, 
$X'(S'\times_S T)^H$ is canonically isomorphic to 
$Y(T)^H$. Suppose further that $H$ is finite.
Let $Y^H$ be the scheme of fixed points under 
$H$ (see e.g. \cite{Bas}, \S 3). Then by definition 
$Y(T)^H=Y^H(T)$. 
\medskip

Let $S=\Spec\cO_K$, $S'=\Spec\cO_L$,
$S_m=\Spec(\cO_K/\pi^{m+1}\cO_K)$  and $S'_m=S'\times_S S_m$. 

\begin{proposition}\label{lift-equiv} Let $\cZ'$ be a 
flat quasi-projective scheme over $S'$ endowed with an equivariant
action of $G=\mathrm{Gal}(L/K)$. 
Then the following properties hold. 
\begin{enumerate}[{\rm (1)}]  
\item Let $\cZ=\cZ'/G$. Then the canonical map $\cZ(S)\to \cZ'(S')^G$ 
is bijective. 
\item Suppose $S'\to S$ is \'etale. 
Then the canonical morphism $\cZ'\to\cZ\times_S S'$ is an isomorphism and 
the canonical morphism $\cZ\to (R_{S'/S}\cZ')^G$ is an
isomorphism. 
\item Suppose that $K$ is henselian, $\cZ'$ is smooth over $S'$ and 
$L/K$ is tamely ramified. Then the canonical map 
\[\cZ'(S')^G\to \cZ'(S'_m)^G\] 
is surjective for all $m\ge 0$. 
\item Suppose that $K$ is henselian and that $\cZ'_L$ is smooth over $L$.
Then there exist $m_0, r_0\ge 0$ such that for all $m\ge m_0$,
and for any $t_m\in \cZ'(S'_m)^G$, the image of 
$t_m$ in $\cZ'(S'_{m-r_0})^G$ lifts to a section in 
$\cZ'(S')^G$. 
\end{enumerate}
\end{proposition}

\begin{proof} (1) First notice that the quotient $\cZ'/G$ exists
because $\cZ'$ is quasi-projective over $\cO_L$. The canonical morphism 
$\cZ'_L\to (\cZ_K)_L$ is an isomorphism by Lemma~\ref{speiser}. 
The canonical map 
\[\cZ'(\cO_L)^G\to \cZ(\cO_K)\subseteq \cZ_K(K)\] 
is injective. Conversely, any section in $\cZ(\cO_K)$ induces
a rational point in $\cZ_K(K)=\cZ'_L(L)^G\subseteq \cZ'_L(L)$. 
The valuative criterion of properness for $\cZ'\to \cZ$ 
implies that the point in $\cZ'_L(L)$ we obtain actually belongs 
to $\cZ'(\cO_L)\cap \cZ'_L(L)^G=\cZ'(\cO_L)^G$. Therefore
$\cZ'(\cO_L)^G\to \cZ(\cO_K)$ is surjective. 

(2) The canonical morphism $\cZ'\to \cZ\times_S S'$ is an isomorphism
by Proposition~\ref{different}. For any $\cO_K$-module $M$ with trivial
action of $G$, the canonical map $M\to (M\otimes_{\cO_K}\cO_L)^G$ is
an isomorphism (use a normal basis of $\cO_L/\cO_K$).  
For any $S$-scheme $T$, the canonical map 
\[\cZ(T)\to R_{S'/S}\cZ'(T)^G=(\cZ'(T\times_S S'))^G=(\cZ(T\times_S S'))^G
=\cZ(T)\] 
is bijective. So $\cZ\to (R_{S'/S}\cZ')^G$ is an isomorphism. 

(3) Let $\cY=R_{S'/S}\cZ'$. We saw above that $\cZ'(S')^G\to \cZ'(S'_m)^G$
can be identified with the canonical map 
\[\cY^G(S)\to \cY^G(S_m).\] 
Let $I\subset G$ be the inertia group, let $L_1=L^I$, $H=G/I$ and let
$S^t=S'/I$. Denote by $\cZ^t=R_{S'/S^t}\cZ'$. It is smooth 
over $S^t$ (\cite{BLR}, Proposition 7.6/5) as well as 
$\cZ_1:=(\cZ^t)^I$ (\cite{Bas}, Proposition 3.4). Let 
$T$ be an $S$-scheme with trivial action of $G$. 
Then 
\[\cZ'(T\times_S S')^G=(\cZ'(T\times_S S')^I)^{H}
=(\cZ_1(T\times_S S^t))^{H}.\] 
Let $\cZ_2=\cZ_1/H$. By (2), $\cZ_2$ is smooth over $S$ and 
$(\cZ_1(T\times_S S^t))^{H}=\cZ_2(T)$. Thus
$\cZ'(T\times_S S')^G=\cZ_2(T)$. As $\cO_K$ is henselian and $\cZ_2$ is
smooth, $\cZ_2(S)\to \cZ_2(S_m)$ is surjective and (3) is proved. 

(4) Applying (2) to $\Spec L\to \Spec K$, we see that $\cZ_K$ is smooth 
over $K$ and $(\cY^G)_K=(\cY_K)^G=\cZ_K$. Our statement then results 
from Elkik's approximation theorem (\cite{Elkik}, Corollaire 1, page 567)
and the identity $\cY^G(T)=\cZ'(T\times_S S')^G$ for all $S$-schemes $T$. 
\end{proof}

\begin{remark} Keep the notation of Proposition~\ref{lift-equiv}. 
\begin{enumerate}
\item 
If $K$ is henselian, $L/K$ is tamely ramified and $\cZ'$ is smooth, 
it is probably true that the canonical map
\[\cZ'(\cO_L)^G\to \cZ'(\Spec (\cO_L/\pi_L^{m+1}\cO_L))^G\] 
is surjective for all $m\ge 0$. Note that the right-hand side is not
$\cZ'(S'_{m})^G$. 
\item The constants $m_0, r_0$ in \ref{lift-equiv}(4) depend on the 
scheme $\cZ'$. When the latter is smooth over $S'$, 
it is probably true that one can find bounds $m_0, r_0$ depending 
only on $v_K(\D_{L/K})$.
\end{enumerate}
\end{remark} 

Next we give an explicit bound on the constants $m_0, r_0$ 
of \ref{lift-equiv} (4) for abelian varieties. 

\begin{proposition}\label{G-approx} Suppose $K$ is complete with 
$\car(K)=0$ and residue characteristic $p\ge 0$. 
Let $h=2[v_K(\D_{L/K})]$ and let 
\[m+1 > h+v_K(p)-1+\dfrac{v_K(p)}{p-1}.\]
($m\ge 0$ if $p=0$.) Let $A$ be an abelian variety over $K$ and 
let $\cA'$ be its N\'eron model over $\cO_L$. 
Then for any $G$-equivariant section
$t_m\in\cA'(\cO_L/\pi^{m+1}\cO_L)^G$
of $\cA'$, there exists a $G$-equivariant section in 
$\cA'(\cO_L)^G=A(K)$ whose image in 
$\cA'(\cO_L/\pi^{m+1-h}\cO_L)$ coincides with that of $t_m$.
\end{proposition}

\begin{proof}Let $\widehat{\cA'}$ be the formal group over $\cO_L$
attached to $\cA'$. Let $r$ be the smallest integer $\ge h/v_K(p)$. 
For all integers $n> rv_K(p)\ge 0$, we have a 
canonical commutative diagram with exact horizontal lines 
\[\xymatrix{ 
  0  \ar[r] &  \widehat{\cA'}(\pi^n \cO_L) \ar[d]\ar[r] & 
\cA'(\cO_L)\ar[d]_{\mathrm{id}} \ar[r] &  \cA'(\cO_L/\pi^n\cO_L) 
\ar[d]\ar[r]&
  0  \\
  0  \ar[r] &  \widehat{\cA'}(\pi^{n}p^{-r} \cO_L) \ar[r] & 
\cA'(\cO_L)\ar[r] &  \cA'(\cO_L/\pi^{n}p^{-r}\cO_L) \ar[r]&
  0.  \\
}
\]
Taking Galois cohomology, we get 
\[\xymatrix{ 
\cA'(\cO_L)^G \ar[r]\ar[d]_{\mathrm{id}} &  
\cA'(\cO_L/\pi^{n}\cO_L)^G \ar[r]\ar[d] & 
\rH^1(G, \widehat{\cA'}(\pi^n \cO_L)) \ar[d]_{f_{n,r}}\\ 
\cA'(\cO_L)^G \ar[r] &  \cA'(\cO_L/\pi^{n}p^{-r}\cO_L)^G \ar[r] & 
\rH^1(G, \widehat{\cA'}(\pi^{n}p^{-r} \cO_L)). \\ 
}
\]
So it is enough to show that $f_{n,r}=0$ when $n>rv_K(p)+v_K(p)/(p-1)$. 
By general results on the formal groups of abelian varieties
(see \cite{LT}, \S 1, \cite{Honda}, Theorem 1, or 
\cite{Tate-p}, \S 2.4, p. 196),
$\widehat{\cA'}(\pi^{\ell} \cO_L)$ is canonically
isomorphic to $\pi^\ell\cO_L$ for all $\ell\ge v_K(p)/(p-1)$. 
So the canonical map $f_{n,r}$ can be identified with
the multiplication-by-$p^r$ map on $\rH^1(G, \cO_L)$. This is
the zero map by Proposition~\ref{2.1.6}(2), thus $f_{n,r}=0$. 
As $rv_K(p)\le h+v_K(p)-1$, the proposition is proved. 
\end{proof}

\begin{corollary}\label{lift-equivariant-sections}
Let $K, K_o$ be henselian discrete valuation fields. 
Let $E, E_o$,  $L, L_o, G$ and $\cW', \cW_o'$ be as at the
beginning of \S \ref{minimal-w_n}. 
\begin{enumerate}[{\rm (1)}] 
\item Let $\epsilon', \epsilon'_o$ be the unit sections of 
$\cW', \cW'_o$ respectively. Then there exists 
an integer $r$ such that for any $m\ge r$, if there are
compatible $G$-equivariant isomorphisms 
\[
\theta_{m} : \cO_{L, m} \simeq \cO_{L_o, m}, \quad 
f_m : \cW'_{m}\simeq \cW'_{o, m},\] 
then there exists a $G$-equivariant isomorphism 
$\cW'_{m-r}\simeq \cW'_{o, m-r}$, compatible with 
the isomorphism 
$\cO_{L, m-r}\simeq \cO_{L_o, m-r}$ and  which maps
$\epsilon'_{m-r}$ to $\epsilon'_{o, m-r}$.  
\item Suppose that $K$ is complete of characteristic $0$. Then
one can take $r=2[v_K(\D_{L/K})]$ 
provided $m$ satisfies the inequality of Proposition~\ref{G-approx}. 
\end{enumerate} 
\end{corollary}

\begin{proof} The image of $\epsilon'_{o, m}$ by $f_m^{-1}$ is a
$G$-equivariant section of $\cW'_m$. Let $r$ be the maximum
of $m_0, r_0$ given in Proposition~\ref{lift-equiv}\ (4). 
Then $f_{m-r}^{-1}(\epsilon'_{o,m-r})\in \cW'_{m-r}(\cO_{L,m-r})$
lifts to a $G$-invariant section $a\in \cW'(\cO_L)$. Let
$t: \cW'\to \cW'$ be the translation by $a$. This is a 
$G$-equivariant isomorphism, and the composition 
$f_{m-r}\circ t_{m-r} : \cW'_{m-r}\to \cW'_{o, m-t}$ is an 
isomorphism compatible with $\cO_{L, m-r}\simeq \cO_{L_o, m-r}$,
and taking $\epsilon'_{m-r}$ to $\epsilon'_{o, m-r}$. 
This proves (1). To prove (2), it is enough to notice that
the smooth locus of $\cW'$ is the neutral component of 
$\mathcal E'$. As $\epsilon'_{o, m}$ is contained in the smooth
locus, we have 
$f_{m}^{-1}(\epsilon'_{o,m})\subset \mathcal \cE'_m(\cO_{L,m})$.
By \ref{G-approx}, $f_{m-r}^{-1}(\epsilon'_{o,m-r})$ lifts to a
section in $\cE'(\cO_L)^G\subseteq \cW'(\cO_L)^G$. The proof
is then achieved as in Part (1). 
\end{proof} 

\end{section}

\affiliationone{
Q. Liu\\
Universit\'e Bordeaux 1\\ 
Institut de Math\'ematiques de Bordeaux \\ 
CNRS, UMR 5251\\
F-33400 Talence,  France\\
\email{Qing.Liu@math.u-bordeaux1.fr}}
\affiliationtwo{
H. Lu\\
Academy of Math. and Systems Science\\
No. 55, Zhongguancun East Road\\ 
Beijing 100190, China
\email{{huajun@amss.ac.cn}}}

\end{document}